\documentclass[reqno]{amsart}
\usepackage{amsfonts}
\usepackage{graphicx}
\usepackage{amsfonts,amsmath, amssymb}
\usepackage{hyperref}
\usepackage{marginnote}
\usepackage{color}
\usepackage{cancel}
\numberwithin{equation}{section}

\newcommand{\dis}{\displaystyle}

\newcommand{\R}{\mathbb{R}}

\newtheorem{theorem}{Theorem}[section]

\newtheorem{lemma}[theorem]{Lemma}
\newtheorem{proposition}[theorem]{Proposition}
\newtheorem{remark}[theorem]{Remark}

\def\v{\varepsilon}

\def\g{\gamma}
\def\d{\delta}

\def\s{\sigma}

\def\f{\frac}

\def\b{\bar}

\usepackage{tikz}

\newcommand{\dd}{{\rm d}}

\renewcommand{\S}{\mathbb{S}}

\newcommand{\Fi}{\mathbf{1}}

\newcommand{\CD}{\mathcal {D}}

\newcommand{\ga}{\gamma}
\newcommand{\om}{\omega}
\newcommand{\la}{\lambda}
\newcommand{\de}{\delta}

\newcommand{\pa}{\partial}

\newcommand{\eps}{\epsilon}

\newcommand{\Ga}{\Gamma}

\newcommand{\vep}{\varepsilon}

\begin{document}

\title[Boltzmann equation with time-periodic boundary]{The Boltzmann equation with time-periodic boundary temperature}

\author[R.-J. Duan]{Renjun Duan}
\address[R.-J. Duan]{Department of Mathematics, The Chinese University of Hong Kong, Hong Kong}
\email{rjduan@math.cuhk.edu.hk}

\author[Y. Wang]{Yong Wang}
\address[Y. Wang]{Institute of Applied Mathematics, Academy of Mathematics and Systems Science, Chinese Academy of Sciences, Beijing 100190, China, and University of Chinese Academy of Sciences}
\email{yongwang@amss.ac.cn}

\author[Z. Zhang]{Zhu Zhang}
\address[Z. Zhang]{Department of Mathematics, The Chinese University of Hong Kong, Hong Kong}
\email{zzhang@math.cuhk.edu.hk}

\begin{abstract}
This paper is concerned with the boundary-value problem on the Boltzmann equation in bounded domains with diffuse-reflection boundary where the boundary temperature is time-periodic. We establish the existence of time-periodic solutions with the same period for both hard and soft potentials, provided that the time-periodic  boundary temperature is sufficiently close to a stationary one which has small variations around a positive constant. The dynamical stability of time-periodic profiles is also proved under small perturbations, and this in turn yields the non-negativity of the profile. For the proof, we develop new estimates in the time-periodic setting. \\
\begin{center}\bf
This paper is dedicated to Professor  Philippe G. Ciarlet on the occasion of his 80th birthday
\end{center}
\end{abstract}

\subjclass[2010]{35Q20, 35B20, 35B35, 35B45}

\keywords{Boltzmann equation, time-periodic boundary, time-periodic solutions, existence, dynamical stability, a priori estimates}
\maketitle

\tableofcontents

\thispagestyle{empty}
\section{Introduction}
Let a rarefied gas be contained in a  bounded domain $\Omega \subset \mathbb{R}^3$ with smooth boundary $\pa\Omega$ on which the diffuse-reflection condition is postulated. We assume that the velocity of the boundary is zero while the temperature of the boundary is periodic in time. One basic problem is to see whether or not there exists a time-periodic motion of such rarefied gas with the same period.

To treat the problem, we assume that the motion of the rarefied gas is governed by the Boltzmann equation
\begin{align}\label{1.1}
\pa_tF+v\cdot\nabla_xF=Q(F,F),\quad t\in\mathbb{R},\ x\in \Omega,\ v\in\mathbb{R}^3.
\end{align}
Here $F=F(t,x,v)\geq 0$ stands for the density distribution function of gas particles with position $x\in \Omega$ and velocity $v\in\R^3$ at time $t\in \R$.
The Boltzmann collision operator $Q(\cdot,\cdot)$ is of the non-symmetric bilinear form:
\begin{align*}
Q(G,F)=&\int_{\mathbb{R}^3}\int_{\mathbb{S}^2} B(v-u,\omega)G(u')F(v')\,\dd\omega\dd u\notag\\
&-\int_{\mathbb{R}^3}\int_{\mathbb{S}^2} B(v-u,\omega)G(u)F(v)\,\dd\omega\dd u.
\end{align*}
Here the relation between the velocity pair $(v',u')$ after collision  with the velocity pair $(v,u)$ before collision for two particles is given by
\begin{equation*}
v' 
=v-[(v-u)\cdot\omega]\omega,\quad 
 u' 
 =u+[(v-u)\cdot\omega]\omega, 
\end{equation*}
with $\om\in \S^2$, satisfying the conservations of momentum and energy due to the elastic collision:
\begin{equation*}
v'+u'=v+u,\quad |v'|^2+|u'|^2=|v|^2+|u|^2.
\end{equation*}
The Boltzmann collision kernel $B(v-u,\om)$ takes the form of
\begin{equation*}
B(v-u,\om)=|v-u|^{\gamma}b(\phi),
\end{equation*}
with
\begin{equation*}
-3<\ga\leq 1,\quad 0\leq b(\phi)\leq C|\cos\phi|,\quad \cos\phi:=\frac{(v-u)\cdot \om}{|v-u|},
\end{equation*}
for a generic constant $C$. Note that the angular cutoff assumption is required and we allow for both hard and soft potentials in the full range.

To solve the Boltzmann equation \eqref{1.1} in the bounded domain, it is supplemented with the following diffuse-reflection boundary condition:
\begin{align}\label{1.2}
F(t,x,v)\big|_{v\cdot n(x)<0}=\mu_{\theta}\int_{u\cdot n(x)>0}F(t,x,u)|u\cdot n(x)|\,\dd u,
\end{align}
for any $t\in \R$, where $n(x)$ denotes the outward normal vector at the boundary point $x\in \pa\Omega$, and $\mu_\theta$ takes the form of
\begin{equation}
\label{def.muth}
\mu_{\theta}:=\mu_{\theta(t,x)}(v)=\f{1}{2\pi\theta^2(t,x)}e^{-\f{|v|^2}{2\theta(t,x)}}.
\end{equation}
Here we have assumed that the boundary velocity is zero and the boundary temperature is a function $\theta(t,x)$ which is periodic in time and may also depend on the space variable.

Throughout this paper, we assume that $\Omega =\{x:\xi (x)<0\}\ $is connected and
bounded with $\xi (x)$ being a smooth function in $\R^3$. We assume $\nabla \xi
(x)\neq 0$ at each boundary point $x$ with $\xi (x)=0$. The outward normal vector $n(x)$ is therefore given by
$n(x)=\nabla \xi (x)/|\nabla \xi (x)|$,  
and it can be extended smoothly near $\partial \Omega =\{x:\xi (x)=0\}.$
We define that $
\Omega $ is convex if there exists a constant $c_{\xi }>0$ such that
\begin{equation*}
\sum_{i,j=1}^3\frac{\pa^2\xi}{\pa x_i\pa x_j} (x)\zeta_{i}\zeta_{j}\geq c_{\xi }|\zeta |^{2}
\end{equation*}
for all $x$ such that $\xi (x)\leq 0$ and for all $\zeta=(\zeta_1,\zeta_2,\zeta_3) \in \mathbb{R}^{3}$.
We denote the phase boundary in the space $\Omega \times \mathbb{R}^{3}$ as $
\gamma =\partial \Omega \times \mathbb{R}^{3}$, and split it into the outgoing
boundary $\gamma _{+}$, the incoming boundary $\gamma _{-}$,  and the
singular boundary $\gamma _{0}$ for grazing velocities, respectively:
\begin{align}
\gamma _{+} &=\{(x,v)\in \partial \Omega \times \mathbb{R}^{3}:
n(x)\cdot v>0\}, \nonumber\\
\gamma _{-} &=\{(x,v)\in \partial \Omega \times \mathbb{R}^{3}:
n(x)\cdot v<0\}, \nonumber\\
\gamma _{0} &=\{(x,v)\in \partial \Omega \times \mathbb{R}^{3}:
n(x)\cdot v=0\}.\nonumber
\end{align}

Note that $\mu_{\theta}$ satisfies the boundary condition \eqref{1.2} but may not be a solution to the Boltzmann equation \eqref{1.1} since the boundary temperature $\theta(t,x)$ may have nontrivial variations in $t$ or $x$. When $\theta(t,x)$ is identical to a constant $\theta_0>0$, for instance, without loss of generality we assume $\theta_0=1$ to the end, the global Maxwellian corresponding to \eqref{def.muth} is reduced to
\begin{equation}
\label{def.gm}
\mu=\mu(v):=\f{1}{2\pi}e^{-\f{|v|^2}{2}},
\end{equation}
which satisfies both \eqref{1.1} and \eqref{1.2}.  In such case, there have been extensive studies of existence, large-time behavior and regularity of small-amplitude $L^\infty$ solution around $\mu$ to the initial-boundary value problem on the Boltzmann equation, for instance, \cite{EGKM,EGKM-18,EGM,Guo2, GKTT-IM, K,LYa}. Readers may also refer to references therein for related works.


When $\theta(t,x)$ is a time-independent function $\b{\theta}(x)$ which has a small variation around $\theta_0$, namely, $\sup_{\pa \Omega}|\bar{\theta}-
{\theta_0}|$ is small enough, one may expect that the large-time behavior of solutions to the initial-boundary value problem on the Boltzmann equation is determined by solutions to
the following steady problem
\begin{equation}\label{sp}
\left\{
\begin{aligned}
&v\cdot\nabla_xF=Q(F,F), \quad x\in \Omega,\ v\in \mathbb{R}^3,\\
&F(x,v)\big|_{v\cdot n(x)<0}=\mu_{\b{\theta}(x)}\int_{u\cdot n(x)>0}F(x,u)|u\cdot n(x)|\,\dd u.
\end{aligned}\right.
\end{equation}
Indeed, for hard potentials $0\leq \ga\leq 1$, \cite{EGKM}  established the existence and dynamical stability of a stationary solution 
${F^*(x,v)}$
to \eqref{sp}. Recently, the result of \cite{EGKM}  has been extended in \cite{DHWZ} to the case of soft potentials $-3<\ga<0$. We refer readers to \cite{DHWZ} for extensive discussions on the subject.


In the current work, we consider the case when $\theta(t,x)$ is a general time-space-dependent function  assumed to be periodic in time with period $T>0$ and sufficiently close to $\bar{\theta}(x)$. Under such situation, we shall prove that there exists a unique time-periodic solution $F^{per}(t,x,v)$ around {$F^*(x,v)$ }
with the same period $T$ for the problem \eqref{1.1} and \eqref{1.2}, and further show the dynamical stability of $F^{per}(t,x,v)$ under small perturbations in the sense that the solution $F(t,x,v)$ to the initial-boundary value problem on the Boltzmann equation \eqref{1.1} with initial data
$F(0,x,v)=F_0(x,v)$
and boundary data \eqref{1.2} exists globally in time and is time-asymptotically close to  $F^{per}(t,x,v)$ whenever $F_0(x,v)$ is sufficiently close to $F^{per}(0,x,v)$. Note that the limiting situation $T=0$ for the period of $\theta(t,x)$ is also allowed and this corresponds to the stationary case considered in \cite{EGKM}  and \cite{DHWZ}  as mentioned above. Therefore, the current work can be regarded as an extension of \cite{EGKM,DHWZ} to the time-periodic boundary.


In what follows we state the main results of this paper. Let
\begin{equation}
\label{def.wnot}
w_{q,\beta}(v):=(1+|v|^2)^{\f{\beta}{2}}e^{q|v|^2}
\end{equation}
be the velocity weight function, and let 
${F^*(x,v)}$ be the steady solution to \eqref{sp} corresponding to the stationary boundary temperature $\bar{\theta}(x)$ constructed in \cite{EGKM,DHWZ}. We assume that $F^*(x,v)$ has the same total mass as that of the global Maxwellian $\mu$ in \eqref{def.gm}, i.e.,
$$
\int_{\Omega}\int_{\R^3} [F^\ast(x,v)-\mu(v)]\,\dd v\dd x=0.
$$
To the end, for brevity we shall write $w_{q,\beta}$ as $w$ by ignoring the dependence of $w$ on parameters $q$ and $\beta$. The first result is concerned with the existence of time-periodic solutions of small amplitude.

\begin{theorem}\label{thm1.1}
Let $-3<\gamma\leq 1, 0\leq q<\f18$ and $\beta>\max\{3,3-\gamma\}$. Assume that $\theta(t,x)$ is a time-periodic function with period $T>0$. Then there exist $\delta>0$ and $C>0$ such that if
\begin{equation*}
\de_1:=\sup_{0\leq t\leq T}|\theta(t,\cdot)-\b{\theta}(\cdot)|_{L^\infty(\partial\Omega)}
\leq \delta,\quad \de_2:=|\b{\theta}(\cdot)-1|_{L^\infty(\partial\Omega)}
\leq  \delta,
\end{equation*}
then the Boltzmann equation \eqref{1.1} with the diffuse-reflection boundary \eqref{1.2} admits a unique nonnegative time-periodic solution with the same period $T$:
\begin{align}\label{1.4}
F^{per}(t,x,v)={F^{*}(x,v)}
+\sqrt{\mu(v)}f^{per}(t,x,v)\geq 0,
\end{align}
satisfying
\begin{equation}
\label{add.thmcon}
\int_\Omega\int_{\mathbb{R}^3}f^{per}(t,x,v)\sqrt{\mu(v)}\,\dd v\dd x=0,\quad t\in \R,
\end{equation}
and
\begin{align}\label{1.5}
\sup_{0\leq t\leq T}\|w f^{per}(t)\|_{L^\infty}+\sup_{0\leq t\leq T}|w f^{per}(t)|_{L^\infty(\g)}\leq C\d_1.
\end{align}
Moreover, if $\Omega$ is convex, $\theta(t,x)$ is continuous on $\mathbb{R}\times\partial\Omega$, and $\b{\theta}(x)$ is continuous on $\partial\Omega$, then $F^{per}(t,x,v)$ is also continuous away from the grazing set $\mathbb{R}\times \g_0$.
\end{theorem}

The second result is concerned with the large-time behavior of solutions to the initial-boundary value problem
\begin{equation}
\label{ibvp}
\left\{\begin{aligned}
&\pa_tF+v\cdot \nabla_xF=Q(F,F),\quad t>0,\ x\in\Omega,\ v\in \mathbb{R}^3,\\
&F(t,x,v)\big|_{v\cdot n(x)<0}=\mu_{\theta(t,x)}\int_{u\cdot n(x)>0}F(t,x,u)|u\cdot n(x)|\,\dd u,\\
&F(0,x,v)=F_0(x,v),
\end{aligned}\right.
\end{equation}
whenever $F_0(x,v)$ is around $F^{per}(0,x,v)$ in a sense to be clarified later on.
%

\begin{theorem}\label{thm1.2}
Let $-3<\gamma\leq1$, $0<q<\f18$ and $\beta>\max\{3,3-\gamma\}$. Then there exist constants $\delta'$, $c>0$, $\vep_0>0$ and $C>0$ such that if
$$
\sup_{0\leq t\leq T}|\theta(t,\cdot)-1|_{L^{\infty}(\pa\Omega)}\leq\delta',
$$
and
$F_0(x,v)=F^{per}(0,x,v)+\sqrt{\mu(v)}f_0(x,v)\geq 0$ satisfies
\begin{align}\label{1.7}
\int_{\Omega}\int_{\mathbb{R}^3}f_{0}(x,v)\sqrt{\mu(v)}\,\dd v\dd x=0,
\end{align}
and
\begin{align*}
\|w 
{f_0}\|_{L^{\infty}}\leq \vep_0,
\end{align*}
then the initial-boundary value problem \eqref{ibvp} on the Boltzmann equation
admits a unique global-in-time solution
$$
F(t,x,v)=F^{per}(t,x,v)+\sqrt{\mu(v)}f(t,x,v)\geq 0,\quad t\geq 0,x\in \Omega,v\in\R^3,
$$
satisfying
\begin{equation*}
\int_{\Omega}\int_{\mathbb{R}^3}f(t,x,v)\sqrt{\mu(v)}\,\dd v\dd x=0
\end{equation*}
and
\begin{equation}\label{1.9}
\left\|w
{f}(t)\right\|_{L^{\infty}}+\left|w
{f}(t)\right|_{L^{\infty}(\g)} 
\leq  C e^{-c t^{\rho}}\|{w}
{f_0}\|_{L^{\infty}},
\end{equation}
for all $t\geq 0$, where $\rho>0$ is determined by
\begin{equation}
\label{def.rho}
\rho=\left\{\begin{aligned}
&1\qquad\qquad\qquad\quad\ \text{if}\quad \ga\in [0,1],\\
&{\frac{{2}
}{2+|\gamma|}}
\in (0,1)\quad  \text{if}\quad \ga\in (-3,0).
\end{aligned}\right.
\end{equation}
Moreover, if $\Omega$ is convex, $F_0(x,v)$ is continuous except on $\g_0$ satisfying
\begin{align*}
F_0(x,v)|_{\gamma_-}=\mu_{\theta}(0,x,v) \int_{u\cdot n(x)>0} F_0(x,u) |u\cdot n(x)| \,\dd u,
\end{align*}
and $\theta(t,x) $ is continuous over $\mathbb{R}\times\partial \Omega$, then the solution $F(t,x,v)$ is also continuous in $[0,\infty)\times \{\bar{\Omega}\times \mathbb{R}^{3}\setminus\g_0\}$.
\end{theorem}

\begin{remark}
In the soft potential case $-3<\ga<0$, the time-decay estimate \eqref{1.9} implies that there is no loss of velocity weight in the weighted $L^\infty$ space for the solution compared to the one for initial data, which is different from the recent result \cite{LYa}. We refer readers to \cite{DHWZ} for more details.
\end{remark}

The issue about the time-periodic solutions to the Boltzmann equation has been studied in \cite{U} and \cite{DUYZ}. Particularly, \cite{U} first considered the case where the Boltzmann equation is driven by a time-periodic source term in the whole space. The main idea of \cite{U} is to study the extra time-decay property of the linearized solution operator $U(t)$ and look for the time-periodic solution as a fixed point to an integral equation
\begin{align*}
f(t)=\int_{-\infty}^t U(t-s) N_f(s)\,\dd s,
\end{align*}
where $N_f(\cdot)$ includes both the nonlinear term and the time-periodic inhomogeneous source. The approach of \cite{U} was later applied in \cite{DUYZ} to consider the Boltzmann equation with a  small time-periodic external force. Note that \cite{DUYZ} has to require a strong assumption that the space dimensions are not less than five, and it has remained a big open problem to remove such restriction.

 A similar time-periodic problem on the Vlasov-Poisson-Fokker-Planck system in the whole space was also considered in \cite{DL} when the background density profile is time-periodic around a positive constant, where the proof is based on another approach different from \cite{U}. It should be pointed out that three space dimensions are allowed in \cite{DL} due to the exponential time-decay structure of the linearized system.


In the current work, we carry out a proof of existence of time-periodic solutions which is different from \cite{DL,DUYZ,U} mentioned above but is similar to the one in \cite{DHWZ} for the steady problem. In fact, instead of solving the Cauchy problem, the basic idea in the present paper is to regard the time-periodic problem as a special boundary value problem over $[0,T]\times\Omega\times\mathbb{R}^3$, with the time-periodic boundary condition at $t=0$ and $t=T$. For the proof, we develop new estimates in the time-periodic setting.

In the end we remark that motivated by the works \cite{AKFG} and \cite{TA}, the existence and dynamical stability of time-periodic profiles to the Boltzmann equation in a bounded interval recently have been also established in \cite{DZ-moving} in the case when one boundary point moves with a small time-periodic velocity. Compared to the current work in the case when the boundary temperature is time-periodic, the mathematical analysis in \cite{DZ-moving} is much harder, since the reformulated problem is related to the Boltzmann equation with a time-periodic external force in the bounded domain.

 The rest of this paper is organized as follows. In Section 2, we make a list of basic lemmas which will be used in the later proof. Then, Section 3 and Section 4 are devoted to the proof of Theorem \ref{thm1.1} and Theorem \ref{thm1.2}, respectively.

\medskip
\noindent{\it Notations.}  Throughout this paper, $C$ denotes a generic positive constant which may vary from line to line.  $C_a,C_b,\cdots$ denote the generic positive constants depending on $a,~b,\cdots$, respectively, which also may vary from line to line. $A\lesssim B$ means that there exists a constant $C>0$ so that $A\leq C B$ and $A\lesssim_{a}B$ means that the constant depends on $a$.
 $\|\cdot\|_{L^2}$ denotes the standard $L^2(\Omega\times\mathbb{R}^3_v)$-norm and $\|\cdot\|_{L^\infty}$ denotes the $L^\infty(\Omega\times\mathbb{R}^3_v)$-norm. We denote $\langle\cdot,\cdot\rangle$ as the inner product in $L^2(\Omega\times \mathbb{R}^3_v)$ or $L^2(\mathbb{R}^3_v)$. Moreover, we define $\|\cdot\|_{L^{2}([0,T];L^2)}=\big\|\|\cdot\|_{L^2}\big\|_{L^{2}[0,T]}$. For the phase boundary integration, we define $d\gamma\equiv |n(x)\cdot v| dS(x)dv$, where $dS(x)$ is the surface measure and define $|f|_{L^p}^p=\int_{\gamma}|f(x,v)|^pd\gamma$ and the corresponding space is denoted as $L^p(\partial\Omega\times\mathbb{R}^3)=L^p(\partial\Omega\times\mathbb{R}^3;d\gamma)$. Furthermore, we denote $|f|_{L^p(\gamma_{\pm})}=|f\Fi_{\gamma_{\pm}}|_{L^p}$ and $|f|_{L^\infty(\gamma_{\pm})}=|f\Fi_{\gamma_{\pm}}|_{L^\infty}$. For simplicity, we denote $|f|_{L^\infty(\g)}=|f|_{L^\infty(\g_+)}+|f|_{L^\infty(\g_-)}$.

\section{Preliminaries}

Recall (cf.~\cite{Gl}) that around the global Maxwellian $\mu$ as in \eqref{def.gm}, one can write
\begin{equation*}
\frac{1}{\sqrt{\mu}}Q(\mu+\sqrt{\mu}f,\mu+\sqrt{\mu} f)=-Lf+\Ga (f,f),
\end{equation*}
where $L$ and $\Ga(\cdot,\cdot)$ are the corresponding linearized operator and nonlinear operator respectively given by
\begin{align*}
Lf=-\f1{\sqrt{\mu}}\Big\{Q(\mu,\sqrt{\mu}f)+Q(\sqrt{\mu}f,\mu)\Big\},\nonumber
\end{align*}
and
\begin{align*}
\Gamma(f,g)=\f1{\sqrt{\mu}}Q(\sqrt{\mu}f,\sqrt{\mu}g).
\end{align*}
Moreover, one has $L=\nu-K$, where the velocity multiplication $\nu=\nu(v)$ is defined by
\begin{equation*}
\nu(v)=\int_{\mathbb{R}^3}\int_{\mathbb{S}^2}B(v-u,\omega)\mu(u)\,\dd\om \dd u\sim (1+|v|)^{\gamma},
\end{equation*}
and the  integral operator $K:=K_1-K_2$ is   defined in terms of
\begin{align}
(K_1f)(v)&=\int_{\mathbb{R}^3}\int_{\mathbb{S}^2}B(v-u,\omega)\sqrt{\mu(v)\mu(u)}f(u)\,\dd\om \dd u,\nonumber
\end{align}
and
\begin{align}
(K_2f)(v)&=\int_{\mathbb{R}^3}\int_{\mathbb{S}^2}B(v-u,\omega)\sqrt{\mu(u)\mu(u')}f(v')\,\dd\om \dd u\nonumber\\
&\quad+\int_{\mathbb{R}^3}\int_{\mathbb{S}^2}B(v-u,\omega)\sqrt{\mu(u)\mu(v')}f(u')\,\dd\om \dd u.\nonumber
\end{align}

\begin{lemma}[\cite{Guo-03,Guo2}]
The operator $L$ is self-adjoint and non-negative. The kernel of $L$ is a five-dimensional space spanned by the following bases: $$
e_0=(2\pi)^{-\f14}\sqrt{\mu};\quad e_i=(2\pi)^{-\f14}v_i\sqrt{\mu},\quad i=1,2,3;\quad e_4=\f{(2\pi)^{-\f14}}{\sqrt{6}}(|v|^2-3)\sqrt{\mu}.
$$
Define the projection $P$ by
\begin{align}\label{P}
Pf=\sum_{i=0}^4\langle f,e_i\rangle e_i.
\end{align}
Then there exists a constant $c_0>0$ such that
\begin{align}\label{c}
\langle Lf,f\rangle\geq c_0|\nu^{1/2}(I-P)f|_{L^{2}(\mathbb{R}^3)}^2.
\end{align}
\end{lemma}

Note that the integral operator $K$ can be written as
\begin{align*}
Kf(v)=\int_{\mathbb{R}^3}k(v,\eta)f(\eta)\,\dd \eta,
\end{align*}
with a symmetric kernel $k(v,\eta)$.
As in \cite{Guo-03,GS}, we introduce a smooth cutoff function $0\leq\chi_m\leq 1$ with $0<m\leq 1$ such that
\begin{equation}
	\chi_m(s)=1~~\mbox{for}~s\leq m;~~~\chi_m(s)=0~~\mbox{for}~s\geq2m.\notag\nonumber
\end{equation}
Then we define
\begin{align}
	(K^mg)(v)&=\int_{\mathbb{R}^3}\int_{\mathbb{S}^2}B(v-u,\omega)\chi_m(|v-u|)\sqrt{\mu(u)\mu(u')}f(v')\,\dd\om \dd u\nonumber\\
	&\quad+\int_{\mathbb{R}^3}\int_{\mathbb{S}^2}B(v-u,\omega)\chi_m(|v-u|)\sqrt{\mu(u)\mu(v')}f(u')\,\dd\om \dd u\nonumber\\
	&\quad-\int_{\mathbb{R}^3}\int_{\mathbb{S}^2}B(v-u,\omega)\chi_m(|v-u|)\sqrt{\mu(v)\mu(u)}f(u)\,\dd\om \dd u\nonumber\\
	&=
	K_2^mf(v)-K^m_1f(v),\nonumber
\end{align}
and
$K^c=K-K^m$.
Correspondingly, one can write
\begin{align}
(K^mf)(v)=\int_{\mathbb{R}^3} k^m(v,\eta) f(\eta)\,\dd\eta,\quad
(K^cf)(v)=\int_{\mathbb{R}^3} k^c(v,\eta) f(\eta)\,\dd\eta.\nonumber
\end{align}
The following estimates on $K^m$ and $K^c$ can be found in \cite{DHWY}.

\begin{lemma}
	Let  $-3<\gamma\leq 1$. Then, for any $0<m\leq1$, it holds that
\begin{equation}\label{2.2}
		|(K^mg)(v)|\leq Cm^{3+\gamma}e^{-\f{|v|^2}{6}}\|g\|_{L^\infty},
\end{equation}
where  $C$ is a generic constant independent of $m$.
The kernels $k^m(v,\eta)$ and $k^c(v,\eta)$ satisfy that for $0\leq a\leq 1$,
\begin{align*}
|k^m(v,\eta)|\leq C
\Big\{|v-\eta|^\gamma+|v-\eta|^{-\f{3-\gamma}2}\Big\}e^{-\frac{|v|^2+|\eta|^2}{16}},
\end{align*}
and
\begin{align}\label{2.4}
|k^c(v,\eta)|&\leq\f{C
m^{a(\gamma-1)}}{|v-\eta|^{1+\f{(1-a)}{2}(1-\gamma)}}\f{1}{(1+|v|+|\eta|)^{a(1-\gamma)}}e^{-\f{|v-\eta|^2}{10}}e^{-\f{||v|^2-|\eta|^2|^2}{16|v-\eta|^2}}\nonumber\\
&\quad+C|v-\eta|^\gamma [1-\chi_m(|v-\eta|)] e^{-\f{|v|^2}{4}}e^{-\f{|\eta|^2}{4}},
\end{align}	
where
$C$ is a generic constant independent of $m$ and $a$.
\end{lemma}

Particularly, since the constant $C$ in 
\eqref{2.4} does not depend on $a\in[0,1]$,  we have the following estimates on $k^c(v,\eta)$ 
by taking $a=1$ and $a=0$. 

\begin{lemma}[\cite{DHWY}]
Let $-3<\gamma\leq 1$. One has
\begin{align}\label{2.5}
|k^c(v,\eta)|&\leq\f{C
m^{\gamma-1}}{|v-\eta|(1+|v|+|\eta|)^{1-\gamma}}e^{-\f{|v-\eta|^2}{10}}e^{-\f{||v|^2-|\eta|^2|^2}{16|v-\eta|^2}},
\end{align}	
and
\begin{align}\label{2.7}
|k^c(v,\eta)|&\leq 
C |v-\eta|^\gamma e^{-\f{|v|^2}{4}}e^{-\f{|\eta|^2}{4}}
+C 
|v-\eta|^{-\f{3-\gamma}2}e^{-\f{|v-\eta|^2}{10}}e^{-\f{||v|^2-|\eta|^2|^2}{16|v-\eta|^2}}.
\end{align}		
Moreover,  it holds that
\begin{equation}\label{2.8}
\int_{\mathbb{R}^3}|k^c(v,\eta)|\cdot \frac{(1+|v|^2)^{\frac{\beta}{2}}e^{q|v|^2}}{(1+|\eta|^2)^{\frac{\beta}{2}}e^{q|\eta|^2}}\,\dd\eta\leq C
m^{\gamma-1}(1+|v|)^{\gamma-2},
\end{equation}
and 	
\begin{align*}
\int_{\mathbb{R}^3}|k^c(v,\eta)|\cdot \frac{(1+|v|^2)^{\frac{\beta}{2}}e^{q|v|^2}}{(1+|\eta|^2)^{\frac{\beta}{2}}e^{q|\eta|^2}}\,\dd\eta\leq C
(1+|v|)^{-1},
\end{align*}
where $\beta\ge0$ is an arbitrary positive constant and $0\leq q<
1/8$. Here the constant $C$ in all estimates above is independent of $m$.
\end{lemma}

In what follows we recall the back-time trajectory in phase space with respect to the diffuse-reflection boundary condition \eqref{1.2} which was first introduced in \cite{Guo2}. First of all, for each boundary point $x\in \pa\Omega$, we define the velocity space for the outgoing particles:
\begin{equation}
\mathcal{V}(x)=\{v'\in\mathbb{R}^3:~v'\cdot n(x)>0\},\nonumber
\end{equation}
associated with the probability measure $\dd\s=\dd\s(x):=
\mu(v')|v'\cdot n(x)|\,\dd v'$.
Given $(t,x,v)$, let $[{X}(s; t,x,v),V(s;t,x,v)]$ be the backward bi-characteristics for the Boltzmann equation, which is determined by
\begin{align}
\begin{cases}
\dis \frac{{\dd}
{X}(s; t,x,v)}{{\dd}
s}=V(s; t,x,v),\\[2mm]
\dis \frac{{\dd}
V(s; t,x,v)}{{\dd}
s}=0,\\[2mm]
[X(t; t,x,v),V(t; t,x,v)]=[x,v].\nonumber
\end{cases}
\end{align}
The solution is then given by
\begin{align}
[X(s;t,x,v),V(s;t,x,v)]=[x-v(t-s),v].\nonumber
\end{align}
For each $(x,v)$ with $x\in \bar{\Omega}$ and $v\neq 0,$ we define the {\it backward exit time} $t_{\mathbf{b}}(x,v)\geq 0$ to be the last moment at which the
back-time straight line $[X(s;0,x,v),V(s;0,x,v)]$ remains in $\bar{\Omega}$:
\begin{equation}
t_{\mathbf{b}}(x,v)=\inf \{\tau \geq 0:x-v\tau\notin\bar{\Omega}\}.\nonumber
\end{equation}
We therefore have $x-t_{\mathbf{b}}{v}\in \partial \Omega $ and $\xi (x-t_{\mathbf{b}}v)=0.$ We also define
\begin{equation}
x_\mathbf{b}(x,v)
=x-t_{\mathbf{b}}v\in \partial \Omega .\nonumber
\end{equation}
Note that $v\cdot n(x_{\mathbf{b}})=v\cdot n({x}_{\mathbf{b}}(x,v)) \leq 0$ always holds true. Let $x\in \bar{\Omega}$, $(x,v)\notin \gamma _{0}\cup \g_{-}$ and
$
(t_{0},x_{0},v_{0})=(t,x,v)$. For $v_{k+1}\in {\mathcal{V}}_{k+1}:=\{v_{k+1}\cdot n({x}_{k+1})>0\}$, the back-time cycle is defined as
\begin{equation}
\left\{\begin{aligned}
X_{cl}(s;t,x,v)&=\sum_{k}\Fi_{[t_{k+1},t_{k})}(s)\{x_{k}-v_k(t_{k}-s)\},\\[1.5mm]
V_{cl}(s;t,x,v)&=\sum_{k}\Fi_{[t_{k+1},t_{k})}(s)v_{k},\nonumber
\end{aligned}\right.
\end{equation}
with
\begin{equation}
({t}_{k+1},{x}_{k+1},v_{k+1})
=({t}_{k}-{t}_{\mathbf{b}}({x}_{k},v_{k}), {x}_{\mathbf{b}}({x}_{k},v_{k}),v_{k+1}).\nonumber
\end{equation}
Define the near-grazing set of $\gamma_{+}$  as
\begin{align}\label{cut}
\gamma^{\v'}_{+}=\left\{(x,v)\in\gamma_{+}:~ |v\cdot n(x)|<{\v'}~\mbox{or}~|v|\geq{\v'}~\mbox{or}~|v|\leq\f1{\v'}\right\}.
\end{align}
Then we have

\begin{lemma}[\cite{Guo2}]\label{lemT}
Let $\v'>0$ be a small positive constant, then it holds that
\begin{multline*}
\int_0^t |f(\tau)\Fi_{\gamma_+\setminus \gamma_{+}^{\v'}}|_{L^1(\ga)}\dd\tau\\
\leq C_{\v',\Omega} \bigg\{\|f(0)\|_{L^1}+\int_0^t \Big[\|f(\tau)\|_{L^1}+\|[\partial_{\tau}+v\cdot\nabla_x] f(\tau)\|_{L^1}\Big]\dd \tau\bigg\},
\end{multline*}
where the positive constant $C_{\v',\Omega}>0$ depends only on $\v'$ and $\Omega$.
\end{lemma}

In the end we conclude this section with the following iteration lemma which will be crucially used later on. The proof of this lemma can be found in \cite{DHWZ}.

\begin{lemma}
Let  $\{a_i\}_{i=0}^\infty $ be a sequence with each $a_i\geq0$. 
For an integer $k\geq 0$, we define a new sequence $\{A_i^k\}_{i=0}^\infty$ by
$$
A_i^k=\max\{a_i, a_{i+1},\cdots, a_{i+k}\},\quad i=0,1,\cdots.
$$

\begin{itemize}
  \item[(i)] Let $D\geq0$ be a constant.  If
  $$
  a_{i+1+k}\leq \f18 A_i^{k}+D, \quad i=0,1,\cdots,
  $$
  then it holds that
\begin{equation}\label{A.1}
A_i^k\leq \left(\f18\right)^{\left[\frac{i}{k+1}\right]}\cdot\max\{A_0^k, \ A_1^k, \cdots, \ A_k^k \}+\f{8+k}{7} D,
\end{equation}
for any $i\geq k+1$.

  \item[(ii)] Let $0\leq \eta<1$ with $\eta^{k+1}\geq\frac14$.  If
$$
  a_{i+1+k}\leq \f18 A_i^{k}+C_k \cdot \eta^{i+k+1},\quad i=0,1,\cdots
$$
then it holds that
\begin{align}\label{A.1-1}
A_i^k\leq \left(\f18\right)^{\left[\frac{i}{k+1}\right]}\cdot\max\{A_0^k, \ A_1^k, \cdots, \ A_k^k \}+2C_k\f{8+k}{7} \eta^{i+k},
\end{align}
for any $i\geq k+1$.
\end{itemize}

\end{lemma}

\section{Existence of time-periodic solutions}

\subsection{Linear problem}
We start from the following linear  problem with time-periodic inhomogeneous source term and boundary data:
\begin{align}\label{3.0.1}
\begin{cases}
\pa_tf+ v\cdot \nabla_xf+Lf=g,\\
f(t,x,v)|_{\gamma_{-}}=P_{\gamma}f+r.
\end{cases}
\end{align}
Here the boundary operator $P_{\g}$ is defined by
\begin{equation}
P_{\gamma}f(t,x,v)=\sqrt{\mu(v)}\int_{v'\cdot n(x)>0} f(t,x,v') \sqrt{\mu(v')} |v'\cdot n(x)|\,\dd v'.\nonumber
\end{equation}
Both the inhomogeneous terms $g=g(t,x,v)$ and $r=r(t,x,v)$ are periodic in time with period $T>0$. Recall the weight function \eqref{def.wnot} and we write
$w(v)=w_{q,\beta}(v)$ for brevity. We define
\begin{equation}
h(t,x,v)=w(v) f(t,x,v)\nonumber.
\end{equation}
Then the equation for $h$ reads:
\begin{align*}
\begin{cases}
\pa_th+v\cdot \nabla_x h+\nu(v) h=K_{w} h+wg,\\[2mm]
\dis h(t,x,v)|_{\gamma_-}=\f{1}{\tilde{w}(v)} \int_{v'\cdot n(x)>0} h(t,x,v') \tilde{w}(v') \dd\sigma'+wr(t,x,v),
\end{cases}
\end{align*}
where
$$
\tilde{w}(v)\equiv \f{1}{w(v)\sqrt{\mu(v)}},\quad
K_wh=wK(\f{h}{w}).
$$

The proof of Theorem \ref{thm1.1} heavily relies on the solvability of the linearized time-periodic problem \eqref{3.0.1}.

\begin{proposition}\label{prop3.2}
Let $-3<\gamma\leq 1$, $0\leq q<\f18$ and $\beta>\max\{3,3-\gamma\}$. Assume that $g$ and $r$ are time-periodic functions with period $T>0$, and satisfy the zero-mass condition
\begin{align}\label{3.0.3}
\int_{\Omega}\int_{\mathbb{R}^3}g(t,x,v)\sqrt{\mu(v)}\,\dd v\dd x=\int_{\g_-}r(t,x,v)\sqrt{\mu(v)}
\,{\dd\gamma}=0,
\end{align}
for all $t\in \mathbb{R}$, and $L^{\infty}$ bounds
$$\sup_{0\leq t\leq T}\|\nu^{-1}wg(t)\|_{L^\infty}+\sup_{0\leq t\leq T}|wr(t)|_{L^\infty(\g_-)}<\infty.$$  Then there exists a unique time-periodic solution $f=f(t,x,v)$ with the same period $T$ to the linearized Boltzmann equation \eqref{3.0.1}, such that
$$
\int_{\Omega\times\mathbb{R}^3} f(t,x,v) \sqrt{\mu} \,\dd v\dd x=0
$$
for all $t\in \mathbb{R}$, and
\begin{multline}\label{3.0.4}
\sup_{0\leq t\leq T}\|wf(t)\|_{L^\infty} +\sup_{0\leq t\leq T}{|wf(t)|_{L^\infty(\g)}}\\
\leq C\sup_{0\leq t\leq T} |wr(t)|_{L^\infty(\gamma_-)}+C\sup_{0\leq t\leq T}\|\nu^{-1}wg(t)\|_{L^\infty}.
\end{multline}
Moreover, if  $\Omega$ is convex,  and $g$ is continuous in
$\mathbb{R}\times\Omega\times\mathbb{R}^3$ and {$r$ is continuous in $\mathbb{R}\times\g_-$,} 
then $f(t,x,v)$ is also continuous away from the grazing set $\mathbb{R}\times\gamma_0$.
\end{proposition}

The following two subsections will be devoted to the proof of Proposition \ref{prop3.2}.

\subsection{A priori $L^\infty$ estimate}

To prove Proposition \ref{prop3.2},
we start from
the a priori $L^{\infty}$ estimate on solutions to the following time-periodic problems:
\begin{equation}\label{3.1.1}
\begin{cases}
\pa_th^{i+1}+v\cdot\nabla_x h^{i+1}+(\vep+\nu(v)) h^{i+1}=\lambda K_w^m h^i+\lambda K^c_w h^i +wg,\\[3mm]
\dis h^{i+1}(t,x,v)|_{\gamma_-}=\f{1}{\tilde{w}(v)} \int_{v'\cdot n(x)>0} h^i({t},x,v') \tilde{w}(v') \dd\sigma'+w(v)r({t},x,v),
\end{cases}
\end{equation}
for $i=0,1,2,\cdots$, where $h^0:=h^0(t,x,v)$ is given. Here $0\leq \lambda\leq 1$ and $\vep>0$ are given parameters, and $g(t,x,v)$ and $r(t,x,v)$ are both time-periodic functions with period $T>0$. Before doing that, we need some preparations. The following lemma gives the mild formulation of $h^{i+1}$. As the proof is more or less the same as {\cite[Lemma 24]{Guo2}, }
we omit it for brevity.

\begin{lemma}
Let $0\leq \lambda\leq 1$ and $\vep>0$. For any $t\in [0,T]$, for almost every $(x,v)\in \bar{\Omega}\times \mathbb{R}^3\backslash (\gamma_0\cup \gamma_-) $ and for any $s\leq t$,  
 we have
\begin{equation}
\label{3.1.2}
h^{i+1}(t,x,v)
=\sum_{\ell=1}^{4}J_\ell+\sum_{\ell=5}^{14}\Fi_{\{t_1> s\}} J_\ell
\end{equation}
with
\begin{align*}
&J_1=\Fi_{\{t_1\leq s\}} e^{-(\vep+{\nu}(v))(t-s)} h^{i+1}(s,x-v(t-s),v),\\
&J_2+J_3+J_4=\int_{\max\{{t}_1,s\}}^t e^{-{\nu}(v)(t-\tau)}\Big[\lambda K_w^mh^{i}+\lambda K^c_wh^{i}+wg\Big](\tau,x-v(t-\tau),v)\dd \tau,\\
&J_5=e^{-(\vep+{\nu}(v))(t-t_1)}w(v) r(t_1,x_1,v),
\end{align*}
\begin{align*}
&J_6=\frac{e^{-(\vep+{\nu}(v))(t-t_1)}}{\tilde{w}(v)} \int_{\Pi _{j=1}^{k-1}\mathcal{V}_{j}}
\sum_{l=1}^{k-2} \Fi_{\{t_{l+1}>s\}} w(v_l)r(t_{l+1},x_{l+1},v_{l})\dd \Sigma_{l}({t}_{l+1}),\\
&J_7=\frac{e^{-(\vep+{\nu}(v))(t-{t}_1)}}{\tilde{w}(v)} \int_{\Pi _{j=1}^{k-1}{\mathcal{V}}_{j}} \sum_{l=1}^{k-1} \Fi_{\{{t}_{l+1}\leq s<{t}_l\}} h^{i+1-l}(s,{x}_l-{v}_l({t}_l-s),v_l) \dd\Sigma_{l}(s),
\end{align*}
\begin{multline*}
J_8+J_9+J_{10}=\frac{e^{-(\vep+{\nu}(v))(t-{t}_1)}}{\tilde{w}(v)} \int_{\Pi _{j=1}^{k-1}{\mathcal{V}}_{j}} \sum_{l=1}^{k-1}\int_s^{{t}_l} \Fi_{\{{t}_{l+1}\leq s<{t}_l\}}\\
[\lambda K_w^mh^{i-l}+\lambda K^c_wh^{i-l}+wg](\tau,{x}_l-{v}({t}_l-\tau),v_l) \dd\Sigma_l(\tau),
\end{multline*}
\begin{multline*}
J_{11}+J_{12}+J_{13}=\frac{e^{-(\vep+{\nu}(v))(t-{t}_1)}}{\tilde{w}(v)}\int_{\Pi_{j=1}^{k-1}{\mathcal{V}}_{j}} \sum_{l=1}^{k-1}\int_{{t}_{l+1}}^{{t}_l} \Fi_{\{{t}_{l+1}>s\}} \\
[\lambda K_w^mh^{i-l}+\lambda K^c_wh^{i-l}+wg](\tau,{x}_l-{v}({t}_l-\tau),v_l) \dd\Sigma_l(\tau),
\end{multline*}
\begin{align*}
J_{14}=\frac{e^{-(\vep+{\nu}(v))(t-{t}_1)}}{\tilde{w}(v)} \int_{\Pi _{j=1}^{k-1}{\mathcal{V}}_{j}}  {\Fi_{\{{t}_{k}>s\}}}
 h^{i+1-k}(t_k,{x}_k,v_{k-1}) \dd\Sigma_{k-1}({t}_k).
\end{align*}
Here we have denoted
\begin{align*}
&\dd\Sigma_l(\tau) = \big\{\Pi_{j=l+1}^{k-1}\dd{\sigma}_j\big\}\cdot \big\{\tilde{w}(v_l) e^{-(\vep+{\nu}(v_l))({t}_l-\tau)} \dd{\sigma}_l\big\}\\
&\qquad\qquad\qquad\cdot \big\{\Pi_{j=1}^{l-1} e^{-(\vep+{\nu}(v_j))({t}_j-{t}_{j+1})} \dd{\sigma}_j\big\},\nonumber
\end{align*}
and $\dd\sigma_j=\mu(v_j)\{n(x_j)\cdot v_j\}\dd v_j$.
\end{lemma}

Next, the following lemma is due to \cite{Guo2}, which gives a quantitative smallness estimate on the measure of possible velocities, so that the particle can not reach down the underlying initial plane, in terms of the number of reflection.

\begin{lemma}\label{lem.smbd}
Let $T>0$. Let $n$ be sufficiently large. There exist constants $\hat{C}_1$ and $\hat{C}_2$ independent of $n$ such that for $k=\hat{C}_1(nT)^{\frac54}$ and $(t,x,v)\in[0,T]\times\bar{\Omega}\times\mathbb{R}^3$, it holds that
\begin{align}\label{3.1.3}
\int_{\Pi _{j=1}^{k-1}{\mathcal{V}}_{j}} \Fi_{\{{t}_k>-nT\}}~  \Pi _{j=1}^{k-1} \dd{\sigma} _{j}\leq \left(\frac12\right)^{\hat{C}_2(nT)^{\frac54}}.
\end{align}
\end{lemma}
\begin{proposition}\label{prop3.1}
Let $-3<\gamma\leq 1, \vep>0$, {$0\leq q<1/8$} and $\beta>3$. Assume that $h^i(t,x,v)$ are all time-periodic functions with period $T>0$ and satisfy $$\sup_{0\leq t\leq T}\{\|h^i(t)\|_{L^\infty}+|h^i(t)|_{L^\infty{{(\g)}}}\}<\infty,$$ for $i=0,1,2,\cdots$. Then there exist
two universal constants $C>0$ and $n>1$ large enough, independent of $i,\lambda$ and $\vep$, such that for $k=\hat{C}_2(nT)^{\frac54}$, it holds, for $i\geq k$, that
\begin{align}\label{3.1.4}
\sup_{0\leq t\leq T}&\|h^{i+1}(t)\|_{L^\infty}+\sup_{0\leq t\leq T}|h^{i+1}(t)|_{L^{\infty}{(\g)}}\nonumber\\
&\leq \frac18 \max_{0\leq l\leq k}\{\sup_{0\leq t\leq T} \|h^{i-l}(t)\|_{L^\infty}\}+C \max_{0\leq l\leq k}\left\{\left\|\f{h^{i-l}}{\langle v\rangle^{|\gamma|} w}\right\|_{L^2([0,T];L^2)}\right\}\nonumber\\
&\quad+C\sup_{0\leq t\leq T}\{\|\nu^{-1}wg(t)\|_{L^\infty}+|wr(t)|_{L^\infty(\g_-)}\}.
\end{align}
Here we have denoted $\langle v\rangle:=(1+|v|^2)^{1/2}$. Moreover, if $h^i\equiv h$ for $i=1,2,\cdots$, i.e., $h$ is a solution, then \eqref{3.1.4} is reduced to the following form
\begin{multline}\label{3.1.5}
\sup_{0\leq t\leq T}\|h(t)\|_{L^\infty}+\sup_{0\leq t\leq T}|h(t)|_{L^{\infty}{(\g)}}\\
\leq C\sup_{0\leq t\leq T}\{\|\nu^{-1}wg(t)\|_{L^\infty}+|wr(t)|_{L^\infty(\g_-)}\}+C\left\|\f{h}{\langle v \rangle^{|\gamma|}w}\right\|_{L^2([0,T];L^2)}.
\end{multline}
\end{proposition}

\begin{proof}
Let $s=-nT$ in \eqref{3.1.2} with $n> 1$ large enough such that \eqref{3.1.3} holds true. We first estimate $J_1$. Note that by periodicity, we have $$h^{i+1}(s,x-(t-s)v,v)=h^{i+1}(0,x-(t-s)v,v).$$ Then if $0\leq\gamma\leq1$, $\nu(v)\geq \nu_0>0$ for some constant $\nu_0$. Then it is direct to get
\begin{align}\label{3.1.6}
|J_1|\leq e^{-\nu_0(t+NT)}\sup_{0\leq t\leq T}\|h^{i+1}(t)\|_{L^{\infty}}.
\end{align}
If $-3<\gamma<0$, $\nu(v)\sim(1+|v|)^{\g}$ no longer has a positive lower bound, when $|v|$ is sufficiently large.
In this case we note that
$$
0\leq t_{{\mathbf{b}}}(x,v)\leq \f{d_\Omega}{|v|},
$$
where $d_\Omega:=\sup_{x,y\in
{\Omega}}|x-y|$ is the diameter of $\Omega$. Then for $|v|>\f{d_\Omega}{nT}$, it holds that
$$
t_1-s=t-t_{{\mathbf{b}}}(x,v)
{+nT}>0.
$$
In other words, $J_1$  appears only when the particle velocity $|v|$ is rather small, so that we have
\begin{align}\label{3.1.7}
|J_1|&\leq \Fi_{\{t_1\leq s\}}\Fi_{\{|v|\leq \f{d_{\Omega}}{nT}\}}e^{-\nu(v)(t-s)}\sup_{0\leq t\leq T}\|h^{i+1}(t)\|_{L^{\infty}}\nonumber\\
&\leq \Fi_{\{t_1\leq s\}}\Fi_{\{|v|\leq 1\}}e^{-\nu(v)(t-s)}\sup_{0\leq t\leq T}\|h^{i+1}(t)\|_{L^{\infty}}\notag\\
&\leq Ce^{-\nu_0(t+nT)}\sup_{0\leq t\leq T}\|h^{i+1}(t)\|_{L^{\infty}},
\end{align}
for the suitably large $n$, where for simplicity of notations we have still denoted the strictly positive constant $\nu_0>0$ to be the infimum of $\nu(v)$ over $|v|\leq 1$. For contributions coming from $g$ and $r$, we notice that
$$
\tilde{w}(v)=\frac{1}{\sqrt{2\pi}} \frac{e^{(\frac14-q)|v|^2}}{(1+|v|^2)^{\frac{\beta}{2}}},
$$
so it holds that
\begin{align}\label{3.1.8}
\frac{1}{\tilde{w}(v)}\leq \sqrt{2\pi} (1+|v|^2)^{\frac{\beta}{2}} e^{-(\frac14-q)|v|^2}\leq C e^{-\frac18|v|^2}.
\end{align}
Moreover, 
 we have
\begin{equation}\label{3.1.9}
\left\{\begin{aligned}
&\int_{\Pi _{j=1}^{k-1}{\mathcal{V}}_{j}} e^{\f{5|v_m|^2}{16}}\  \Pi_{j=1}^{k-1} \dd{\sigma}_j\leq C, 
\\
&\int_{\Pi _{j=1}^{k-1}{\mathcal{V}}_{j}} \sum_{l=1}^{k-1} \Fi_{\{{t}_{l+1}\leq s<{t}_l\}} e^{\f{5|v_m|^2}{16}}\ \Pi_{j=1}^{k-1} \dd{\sigma}_j\leq Ck,\\
&\int_{\Pi _{j=1}^{k-1}{\mathcal{V}}_{j}} \sum_{l=1}^{k-1} \Fi_{\{{t}_{l+1}> s\}} e^{\f{5|v_m|^2}{16}} \Pi_{j=1}^{k-1} \dd{\sigma}_j\leq Ck,
\end{aligned}\right.
\end{equation}
for all $1\leq m\leq k-1.$ Combining this with periodicity of $r$ and $g$,  we get that
\begin{equation}\label{3.1.10}
\left\{\begin{aligned}
&|J_4|+|J_{10}|+|J_{13}|\leq Ck\sup_{0\leq t\leq T}\|\nu^{-1}wg(t)\|_{L^{\infty}}, \\ 
&|J_5|+|J_6|\leq Ck\sup_{0\leq t\leq T}|wr(t)|_{L^{\infty}(\g_-)}.
\end{aligned}\right.
\end{equation}
Next, we shall estimate $J_7$. If $0\leq\gamma\leq 1$, we use the fact that $\nu(v)\geq\nu_0>0$ as well as  \eqref{3.1.8} {and \eqref{3.1.9}} to get
\begin{align}\label{3.1.11}
|J_7|&\leq C{e^{-\f18|v|^2}}e^{-\nu_0(t+nT)}\max_{1\leq l\leq k-1}\big\{\sup_{0\leq t\leq T}\|h^{i+1-l}(t)\|_{L^{\infty}}\big\}\notag \\
&\qquad\qquad\qquad\qquad\times\int_{\Pi _{j=1}^{k-1}{\mathcal{V}}_{j}} \sum_{l=1}^{k-1} \Fi_{\{{t}_{l+1}\leq s<{t}_l\}} \tilde{w}(v_l)\ \Pi_{j=1}^{k-1} \dd{\sigma}_j\nonumber\\
&\leq Ck{e^{-\f18|v|^2}}e^{-\nu_0(t+nT)}\max_{1\leq l\leq k-1}\big\{\sup_{0\leq t\leq T}\|h^{i+1-l}(t)\|_{L^{\infty}}\big\}.
\end{align}
If $-3<\gamma<0$, we again note that $\nu(v)$ no longer has a positive lower bound. In this case, it holds from Young's inequality that
$$
\nu(v)(\tau_1-\tau_2)+\f{|v|^2}{16}\geq c (\tau_1-\tau_2)^{\alpha},
$$
for any $\tau_1>\tau_2$, where we have taken
$
\alpha=\f{2}{2+|\gamma|},
$ and
$c>0$ is a constant independent of $\tau_1$, $\tau_2$ and $v$. In the sequel $c>0$ may take different values at different places.  So, from \eqref{3.1.8} we have
\begin{align}
\f{e^{-\nu(v)(t-t_1)}}{\tilde{w}(v)}\leq Ce^{-\f{|v|^2}{16}}e^{-c(t-t_1)^{\alpha}},\nonumber
\end{align}
and
\begin{align}
|J_7|\leq& Ce^{-\f{|v|^2}{16}}e^{-c(t-t_1)^{\alpha}}\max_{1\leq l\leq k-1}\big\{\sup_{0\leq t\leq T}\|h^{i+1-l}(t)\|_{L^{\infty}}\big\}\nonumber\\
&\times\sum_{l=1}^{k-1}\int_{\Pi_{j=1}^{l}{\mathcal{V}}_{j}}\Fi_{\{{t}_{l+1}\leq s<{t}_l\}}\tilde{w}(v_l)e^{-\nu(v_l)(t_l-s)}\dd\sigma_l\Pi_{j=1}^{l-1}e^{-\nu(v_j)(t_j-t_{j+1})}\dd \sigma_j.\nonumber
\end{align}
For each $l$, we take
$
|v_m|=\max\{|v_1|,\cdots,|v_l|\}.
$
Then it holds that
$$
\Pi_{j=1}^{l-1}e^{-\nu(v_j)(t_j-t_{j+1})}\times e^{-\nu(v_l)(t_l-s)}\tilde{w}(v_l)\\
\leq e^{-\nu(v_m)(t_1-s)}e^{\f{|v_m|^2}{4}}\leq e^{-c(t_1-s)^{\alpha}}e^{\f{5|v_m|^2}{16}}.
$$
Thus one has
\begin{align}\label{3.1.12}
|J_7|&\leq e^{-\f{|v|^2}{16}}e^{-c(t-
{s})^{\alpha}}\max_{1\leq l\leq k-1}\big\{\sup_{0\leq t\leq T}\|h^{i+1-l}(t)\|_{L^{\infty}}\big\}\nonumber\\
&\quad\times\sum_{l=1}^{k-1}\sum_{m=1}^l\int_{\Pi_{j=1}^{l}{\mathcal{V}}_{j}}\Fi_{\{{t}_{l+1}\leq s<{t}_l\}}e^{\f{5|v_m|^2}{16}}\Pi_{j=1}^l\dd \sigma_j\nonumber\\
&\leq Ck^2e^{-\f{|v|^2}{16}}e^{-c(t+nT)^{\alpha}}\max_{1\leq l\leq k-1}\big\{\sup_{0\leq t\leq T}\|h^{i+1-l}(t)\|_{L^{\infty}}\big\}.
\end{align}
Here we have used the elementary fact that $a^{\alpha}+b^{\alpha}\geq(a+b)^{\alpha}$ for $a$, $b\geq 0$ and $0\leq \alpha\leq 1$.
For $J_{14}$, it follows from \eqref{3.1.3} and \eqref{3.1.8} that
\begin{align}\label{3.1.13}
|J_{14}|\leq Ce^{-\f{|v|^2}{16}}\left(\f12\right)^{\hat{C}_2(nT)^{\f54}}\sup_{0\leq t\leq T}\|h^{i+1-k}(t)\|_{L^{\infty}}.
\end{align}
For the contribution from $K^m$, we use \eqref{2.2} to obtain
\begin{equation}\label{3.1.14}
|J_2|\leq{Cm^{3+\gamma}w(v)e^{-\f{1}{6}|v|^2}\sup_{0\leq t\leq T}\|h^{i}(t)\|_{L^{\infty}}}\leq Cm^{3+\gamma}e^{-\f{|v|^2}{48}}\sup_{0\leq {t}
\leq T}\|h^{i}(t)\|_{L^{\infty}}.
\end{equation}
Similarly, we use \eqref{2.2}, \eqref{3.1.8} and \eqref{3.1.9} to get
\begin{align}\label{3.1.15}
|J_8|\leq& Cm^{3+\gamma}e^{-\f{|v|^2}{8}}\max_{1\leq l\leq k-1}\{\sup_{0\leq {t}
\leq T}\|h^{i-l}(t)\|_{L^{\infty}}\}\nonumber\\
&\times\int_{\Pi _{j=1}^{k-1}{\mathcal{V}}_{j}} \sum_{l=1}^{k-1} \Fi_{\{{t}_{l+1}\leq s<{t}_l\}}\int_s^{t_l}e^{-\nu(v_l)(t_l-\tau)}\nu(v_l)\dd\tau \nu^{-1}(v_l)\tilde{w}(v_l)\ \Pi_{j=1}^{k-1} \dd{\sigma}_j\nonumber\\
\leq &Ckm^{3+\gamma}e^{-\f{|v|^2}{8}}\max_{1\leq l\leq k-1}\{\sup_{0\leq {t}
\leq T}\|h^{i-l}(t)\|_{L^{\infty}}\},
\end{align}
and
\begin{align}\label{3.1.16}
|J_{11}|\leq& Cm^{3+\gamma}e^{-\f{|v|^2}{8}}\max_{1\leq l\leq k-1}\{\sup_{0\leq {t}
\leq T}\|h^{i-l}(t)\|_{L^{\infty}}\}\nonumber\\
&\times\int_{\Pi _{j=1}^{k-1}{\mathcal{V}}_{j}} \sum_{l=1}^{k-1} \Fi_{\{{t}_{l+1}>s\}} \int_{t_{l+1}}^{t_l}e^{-\nu(v_l)(t_l-\tau)}\nu(v_l)\dd \tau \nu^{-1}(v_l)\tilde{w}(v_l)\ \Pi_{j=1}^{k-1} \dd{\sigma}_j\nonumber\\
\leq &Ckm^{3+\gamma}e^{-\f{|v|^2}{8}}\max_{1\leq l\leq k-1}\{\sup_{0\leq {t}
\leq T}\|h^{i-l}(t)\|_{L^{\infty}}\}.
\end{align}
It remains to estimate the terms involving $K^c$. Firstly, we have
\begin{align}\label{3.1.17}
|J_9| &\leq  C e^{-\f18|v|^2} \sum_{l=1}^{k-1}\int_{\Pi _{j=1}^{l-1}{\mathcal{V}}_{j}} \dd{\sigma}_{l-1}\cdots \dd{\sigma}_1 \int_{\mathcal{V}_l}\int_{\mathbb{R}^3} \int_s^{{t}_l} e^{-\nu(v_l)(t-\tau)} \Fi_{\{{t}_{l+1}\leq s<{t}_l\}} \nonumber\\
&\qquad\qquad\qquad\qquad\times \tilde{w}(v_l) |k^c_w(v_l,v') h^{i-l}(\tau,{x}_l-{{v_l}}({t}_l-\tau),v')|\dd \tau\dd v' \dd{\sigma}_l \nonumber\\
&=C e^{-\f18|v|^2} \sum_{l=1}^{k-1}\int_{\Pi _{j=1}^{l-1}{\mathcal{V}}_{j}} \dd{\sigma}_{l-1}\cdots \dd{\sigma}_1 \int_{\mathcal{V}_l\cap \{|v_l|\geq N\}}\int_{\mathbb{R}^3}\int_s^{{t}_l} (\cdots)\dd \tau\dd v' \dd{\sigma}_l \nonumber\\
&\quad+C e^{-\f18|v|^2} \sum_{l=1}^{k-1}\int_{\Pi _{j=1}^{l-1}{\mathcal{V}}_{j}} \dd{\sigma}_{l-1}\cdots \dd{\sigma}_1 \int_{\mathcal{V}_l\cap \{|v_l|\leq N\}}\int_{\mathbb{R}^3} \int_s^{{t}_l}  (\cdots)\dd \tau\dd v' \dd{\sigma}_l\nonumber\\
&:=\sum_{l=1}^{k-1} (J_{91l}+J_{92l}).
\end{align}
For $J_{91l}$, we use \eqref{3.1.9} to obtain that
\begin{align}\label{3.1.18}
\sum_{l=1}^{k-1}J_{91l}&\leq Ck e^{-\f18|v|^2}\max_{1\leq l\leq k-1}\bigg\{\sup_{0\leq t\leq T}\|h^{i-l}(t)\|_{L^\infty}\int_{\Pi _{j=1}^{l-1}{\mathcal{V}}_{j}} \dd{\sigma}_{l-1}\cdots \dd{\sigma}_1
 \nonumber\\
&\qquad\times\int_{\mathcal{V}_l\cap \{|v_l|\geq N\}} \int_s^{{t}_l} e^{-{\nu}(v_l)(t_l-\tau)}\nu(v_l)
{e^{-\f{|v_l|^2}{32}}\dd\tau e^{\f{5|v_l|^2}{16}}\dd\sigma_l}\bigg\}\nonumber\\
&\leq Ck e^{-\f18|v|^2} e^{-\frac1{32}N^2}\max_{1\leq l\leq k-1}\big\{\sup_{0\leq t\leq T}\|h^{i-l}(t)\|_{L^\infty}\big\}.
\end{align}
For $J_{92l}$, it holds that
\begin{align}
J_{92l}&\leq C e^{-\f18|v|^2} \int_{\Pi _{j=1}^{l-1}{\mathcal{V}}_{j}} \Pi _{j=1}^{l-1}\dd{\sigma}_{j}
\int_{\mathcal{V}_l\cap \{|v_l|\leq N\}}\int_{\mathbb{R}^3} \int_{{t}_l-\frac1N}^{{t}_l}(\cdots)\dd \tau\dd v' \dd{\sigma}_l\nonumber\\
&\quad+C e^{-\f18|v|^2} \int_{\Pi _{j=1}^{l-1}{\mathcal{V}}_{j}} 
\Pi _{j=1}^{l-1}\dd{\sigma}_{j}
\int_{\mathcal{V}_l\cap \{|v_l|\leq N\}} \int_s^{{t}_l-\frac1N} e^{-\nu(v_l)(t_l-\tau)} e^{-\f18|v_l|^2}\dd\tau \dd v_l\nonumber\\
&\qquad\quad\times \int_{|v'|\geq 2N}  |k_w^c(v_l,v')| e^{\frac{|v_l-v'|^2}{64}} \dd v' e^{-\frac{N^2}{64}}\cdot \max_{1\leq l\leq k-1}\{\sup_{0\leq t\leq T}\|h^{i-l}(t)\|_{L^\infty}\}\nonumber\\
&\quad+C e^{-\f18|v|^2} \int_{\Pi _{j=1}^{l-1}{\mathcal{V}}_{j}} 
\Pi _{j=1}^{l-1}\dd{\sigma}_{j}
\int_s^{{t}_l-\frac1N} \dd \tau\int_{\mathcal{V}_l\cap \{|v_l|\leq N\}}\int_{|v'|\leq 2N} \nonumber\\
&\qquad\quad\times  \Fi_{\{{t}_{l+1}\leq s<{t}_l\}} e^{-\frac18|v_l|^2} |k^c_w(v_l,v') h^{i-l}(\tau,x_l-{v_l}
({t}_l-{\tau}
),v')|\dd v' \dd v_l. \nonumber
\end{align}
Then, by \eqref{2.7} we have
\begin{align}
J_{92l}
&\leq C e^{-\f18|v|^2} \int_{\Pi _{j=1}^{l-1}{\mathcal{V}}_{j}} 
\Pi _{j=1}^{l-1}\dd{\sigma}_{j}
\bigg\{\int_s^{{t}_l-\frac1N}\int_{\mathcal{V}_l\cap \{|v_l|\leq N\}}\int_{|v'|\leq 2N} \nonumber\\
&\qquad\quad\times  \Fi_{\{{t}_{l+1}\leq s<{t}_l\}} e^{-\frac18|v_l|^2} |k^c_w(v_l,v') h^{i-l}(\tau,x_l-{v_l}
({t}_l-{\tau}
),v')|\dd v' \dd v_l\dd\tau\bigg\}\nonumber\\
&\qquad+\frac{C}{N}e^{-\f18|v|^2}\cdot\max_{1\leq l\leq k-1}\{\sup_{0\leq t\leq T}\|h^{i-l}{(t)}\|_{L^\infty}\}.
\label{3.1.19}
\end{align}
By H\"older's inequality, the integral term on the right-hand of \eqref{3.1.19} 
\begin{equation}
\label{add.pf1}
\int_s^{{t}_l-\frac1N}\int_{\mathcal{V}_l\cap \{|v_l|\leq N\}}\int_{|v'|\leq 2N}(\cdots)\, \dd v' \dd v_l\dd\tau:=\iiint_{\CD}(\cdots)\, \dd v' \dd v_l\dd\tau
\end{equation}
is bounded  by
\begin{align}
&C_N\left\{
\iiint_{\CD}
e^{-\frac18|v_l|^2} |k^c_w(v_l,v')|^2 \dd v' \dd v_l \dd\tau\right\}^{1/2}\nonumber\\
&\quad \times \left\{
\iiint_{\CD}
\Fi_{\{{t}_{l+1}\leq s<{t}_l\}} \left|\frac{h^{i-l}(\tau,{x}_l-{v_l}
({t}_l-\tau),v')}{\langle v'\rangle^{|\gamma|}w(v')}\right|^2 \dd v' \dd v_l \dd\tau\right\}^{1/2}\nonumber\\
&\leq C_Nn^{1/2} m^{\gamma-1}\bigg\{
\iiint_{\CD}
\Fi_{\{{t}_{l+1}\leq s<{t}_l\}} \left|\frac{h^{i-l}(\tau,{x}_l-{v_l}
({t}_l-\tau),v')}{\langle v'\rangle^{|\gamma|} w(v')}\right|^2 \dd v' \dd v_l\dd\tau \bigg\}^{1/2}.\nonumber
\end{align}
Here we have used \eqref{2.5} in the last inequality. Note that $y_l:={x}_l-{v_l}
({t}_l-\tau)\in \Omega$ for $s\leq \tau\leq t_l-\f1N$. Making change of variables $v_l\rightarrow y_l$, we obtain that \eqref{add.pf1} is bounded by
\begin{align}
 C_Nn^{1/2}m^{\gamma-1}\left\{\int_s^{t_l}\left\|\f{h^{i-l}(\tau)}{\langle v\rangle^{|\gamma|}w}\right\|_{L^2}^2\dd\tau\right\}^{1/2}.\nonumber
\end{align}
We  use periodicity of $h^{i-l}$ to further bound the above term by    
$$
C_Nn^{1/2}m^{\gamma-1}\bigg\{\int_s^T\left\|\f{h^{i-l}(\tau)}{\langle v\rangle^{|\gamma|}w}\right\|_{L^2}^2\dd\tau\bigg\}^{1/2}\leq C_N nm^{\gamma-1}\left\|\f{h^{i-l}}{\langle v\rangle^{|\gamma|}w}\right\|_{L^2([0,T];L^2)}.
$$
Combining this with \eqref{3.1.17}, \eqref{3.1.18} and \eqref{3.1.19}, we get
\begin{align}\label{3.1.20}
|J_9|\leq& \f{Ck}{N}e^{-\f{|v|^2}{8}}\max_{1\leq l\leq k-1 }\{\sup_{0\leq t\leq T}\|h^{i-l}(t)\|_{L^{\infty}}\}\nonumber\\
&+
{C_{N,n,m}}e^{-\f{|v|^2}{8}}\max_{1\leq l\leq k-1}\left\{\left\|\f{h^{i-l}}{\langle v\rangle^{|\gamma|}w}\right\|_{L^2([0,T];L^2)}\right\}.
\end{align}
Similarly, for $J_{12}$ one has
\begin{align}\label{3.1.21}
|J_{12}|\leq& \f{Ck}{N}e^{-\f{|v|^2}{8}}\max_{1\leq l\leq k-1 }\{\sup_{0\leq t\leq T}\|h^{i-l}(t)\|_{L^{\infty}}\}\nonumber\\
&+
{C_{N,n,m}}e^{-\f{|v|^2}{8}}\max_{1\leq l\leq k-1}\left\{\left\|\f{h^{i-l}}{\langle v\rangle^{|\gamma|}w}\right\|_{L^2([0,T];L^2)}\right\}.
\end{align}
Collecting all estimates \eqref{3.1.6}, \eqref{3.1.7}, \eqref{3.1.10}, \eqref{3.1.11}, \eqref{3.1.12}, \eqref{3.1.13}, \eqref{3.1.14}, \eqref{3.1.15}, \eqref{3.1.16}, \eqref{3.1.20} and \eqref{3.1.21}, we get that for $t\in[0,T]$,
\begin{multline}\label{3.1.22}
|h^{i+1}(t,x,v)|\leq \int_{\max\{t_1,s\}}^te^{-\nu(v)(t-\tau)}\int_{\mathbb{R}^3}|k^c_w(v,v')h^i(\tau,x-(t-\tau)v,v')|\dd v'\dd \tau\\
+A_i(t,v),
\end{multline}
where we have denoted
\begin{align}
A_i(t,v):=&C{k^2}e^{-\f{|v|^2}{48}}\bigg\{m^{3+\gamma}+e^{-c(t+nT)^{\alpha}}\notag\\
&\qquad\qquad\quad +2^{-\hat{C}_2(nT)^{\f54}}+\f1N\bigg\}\
\max_{0\leq l\leq k-1}\{\sup_{0\leq t\leq T}\|h^{i-l}(t)\|_{L^{\infty}}\}\nonumber\\
&+Ce^{-c(t+nT)}\sup_{0\leq t\leq T}\|h^{i+1}(t)\|_{L^{\infty}}\notag \\
&+Ck\left\{\sup_{0\leq t\leq T}\|\nu^{-1}wg(t)\|_{L^{\infty}}+\sup_{0\leq t\leq T}|wr(t)|_{L^{\infty}(\g_-)}\right\}\nonumber\\
&+{C_{N,n,m}}e^{-\f{|v|^2}{8}}\max_{0\leq l\leq k-1}\left\{\left\|\f{h^{i-l}}{\langle v\rangle^{|\gamma|}w}\right\|_{L^{2}([0,T];L^2)}\right\},\nonumber
\end{align}
and
\begin{equation}
\label{def.kno}
k=\hat{C}_1(nT)^{\f54}\sim (nT)^{\f{5}{4}}.
\end{equation}
Denoting $x':=x-(t-\tau)v$ and $t_1':=t_1(\tau,x',v')$, we use \eqref{3.1.22} for $h^{i}(\tau,x',v')$ to evaluate
\begin{align}\label{3.1.23}
|h^{i+1}(t,x,v)|\leq &A_i(t,v)+\int_{\max\{t_1,s\}}^te^{-\nu(v)(t-\tau)}\int_{\mathbb{R}^3}|k^c_w(v,v')A_{i-1}(\tau,v')\dd v'\dd \tau\nonumber\\
&+\int_{\max\{t_1,s\}}^te^{-\nu(v)(t-\tau)}\dd\tau\int_{\mathbb{R}^3}\int_{\max\{t_1',s\}}^\tau
\int_{\mathbb{R}^3}\bigg\{e^{-\nu(v')(\tau-\tau')}\nonumber\\
&\quad\times|k^c_w(v,v')k^c_w(v',v'')h^{i-1}(\tau',x'-(\tau-\tau')v',v'')|\bigg\}\dd v'' \dd\tau'\dd v'\nonumber\\
=&A_i(t,v)+B_1+B_2,
\end{align}
where $B_1$ and $B_2$ denote two integral terms on the right-hand respectively.
It follows from \eqref{2.8} that 
\begin{align}\label{3.1.24}
B_1\leq &C{k^2}\left\{m^{\gamma-1}e^{-c nT}+m^{3+\gamma}+e^{-c(t+nT)^{\alpha}}
+2^{-\hat{C}_2(nT)^{\f54}}+\f1N\right\}\notag\\
&\qquad\times\max_{0\leq l\leq k}
\{\sup_{0\leq t\leq T}\|h^{i-l}(t)\|_{L^{\infty}}\}\nonumber\\
&+Ck m^{\gamma-1}\left\{\sup_{0\leq t\leq T}\|\nu^{-1}wg(t)\|_{L^{\infty}}+\sup_{0\leq t\leq T}\|wr(t)\|_{L^{\infty}(\gamma_-)}\right\}\nonumber\\
&+{C_{N,n,m}}\max_{0\leq l\leq k}\left\{\left\|\f{h^{i-l}}{\langle v \rangle^{|\gamma|}w}\right\|_{L^{2}([0,T];L^2)}\right\}.
\end{align}
Finally, we estimate $B_2$. If $|v|>N$, we have from \eqref{2.8} that
\begin{align}\label{3.1.25}
B_2\leq Cm^{2(\gamma-1)}(1+|v|)^{-2}\sup_{0\leq t\leq T}\|h^{i-1}(t)\|_{L^{\infty}}\leq C\f{m^{2(\gamma-1)}}{N^2} \sup_{0\leq t\leq T}\|h^{i-1}(t)\|_{L^{\infty}}.
\end{align}
If $|v|\leq N$, we denote the integrand of $B_2$
as $U(\tau',v',v'';\tau,v)$, and split the integral domain 
 with respect to $\dd\tau'\dd v''\dd v'$ into the following four parts:
\begin{align*}
\cup_{i=1}^4\mathcal{O}_i:=&\{|v'|\geq 2N\}
\cup\{|v'|\leq 2N, |v''|>3N\}\\
&\cup\{|v'|\leq 2N, |v''|\leq 3N, \tau-\f1N\leq \tau'\leq \tau\}\nonumber\\
&\cup\{|v'|\leq 2N, |v''|\leq 3N, \max\{t_1',s\}\leq \tau'\leq \tau-\f1N\}.
\end{align*}
Over $\mathcal{O}_1\cup\mathcal{O}_2$, we have either $|v-v'|\geq N$ or $|v'-v''|\geq N $, so that one of the following is valid:
\begin{equation*}
|k^c_w(v,v')|\leq e^{-\f{N^2}{64}}e^{\f{|v-v'|^2}{64}}|k^{c}_w(v,v')|
\ \text{or}\ 
|k^c_w(v',v'')|\leq e^{-\f{N^2}{64}}e^{\f{|v'-v''|^2}{64}}{|k^c_w(v',v'')|.}
\end{equation*}
Recall \eqref{2.5}. Then it holds that
\begin{equation*}
\int_{\mathbb{R}^3}|k^c_w(v,v')|e^{\f{|v-v'|^2}{64}}\dd v'\leq Cm^{\gamma-1}\nu(v),
\end{equation*}
or 
$$
\int_{\mathbb{R}^3}|k^c_w(v',v'')|e^{\f{|v'-v''|^2}{64}}\dd v''\leq Cm^{\gamma-1}\nu(v').
$$
Therefore one has
\begin{multline}\label{3.1.26}
\int_{\max\{t_1,s\}}^te^{-\nu(v)(t-\tau)}\int_{\mathcal{O}_1\cup\mathcal{O}_2}U(\tau',v',v'';\tau,v)\dd v''\dd\tau'\dd v'\dd \tau\\
\leq Cm^{2(\gamma-1)}e^{-\f{N^2}{64}}\sup_{0\leq t\leq T}\|h^{i-1}(t)\|_{L^{\infty}}.
\end{multline}
Over $\mathcal{O}_3$, it is direct to obtain
\begin{multline}\label{3.1.27}
\int_{\max\{t_1,s\}}^te^{-\nu(v)(t-\tau)}\int_{\mathcal{O}_3}U(\tau',v',v'';\tau,v)\dd v''\dd\tau'\dd v'\dd \tau\\
\leq C\f{m^{2(\gamma-1)}}{N}\sup_{0\leq t\leq T}\|h^{i-1}(t)\|_{L^{\infty}}.
\end{multline}
For $\mathcal{O}_4$, we have, from \eqref{2.5}, that
\begin{align}\label{3.1.28}
&\int_{\mathcal{O}_4}U(\tau',v',v'';\tau,v)\dd v''\dd\tau'\dd v'\nonumber\\
&\leq C_N\bigg\{
\int_{\mathcal{O}_4}
|k^c_w(v,v')k^c_w(v',v'')|^2\dd v'\dd v''\dd\tau'\bigg\}^{\f12}\nonumber\\
&\qquad\qquad\times\bigg\{
\int_{\mathcal{O}_4}
\Fi_{\{\max\{t_1',s\}\leq \tau'\leq \tau\}}\left|\f{h^{i-1}(\tau',y',v'')}{\langle v''\rangle^{|\gamma|} w(v'')}\right|^2\dd v'\dd v''\dd\tau'\bigg\}^{\f12}\nonumber\\
&\leq {C_{N,n,m}}\bigg\{
\int_{\mathcal{O}_4}
\Fi_{\{\max\{t_1',s\}\leq \tau'\leq \tau\}}\left|\f{h^{i-1}(\tau',y',v'')}{\langle v''\rangle^{|\gamma|} w(v'')}\right|^2\dd v'\dd v''\dd\tau'\bigg\}^{\f12},
\end{align}
where we have denoted $y':=y-(\tau-\tau')v'$. Making change of variable $v'\rightarrow y'$, the right-hand side  of \eqref{3.1.28} is further bounded by
$$
C_{N,n,m}\left\{\int_s^T\left\|\f{h^{i-1}(\tau')}{\langle v \rangle^{|\gamma|} w}\right\|^2_{L^2}\dd\tau'\right\}^{1/2}
\leq {C_{N,n,m}}\left\|\f{h^{i-1}}{\langle v\rangle^{|\gamma|} w}\right\|_{L^2([0,T],L^2)}.
$$
Then it holds that
\begin{align}
\int_{\max\{t_1,s\}}^t\int_{\mathcal{O}_4}U(\tau',v',v'';\tau,v)\dd v''\dd\tau'\dd v'\dd\tau\leq {C_{N,n,m}}
\left\|\f{h^{i-1}}{\langle v\rangle^{|\gamma|} w}\right\|_{L^2([0,T],L^2)}.\nonumber
\end{align}
The above estimate
together with \eqref{3.1.25}, \eqref{3.1.26} and \eqref{3.1.27} yield that
\begin{align}
B_2\leq \f{Cm^{2(\gamma-1)}}{N}\sup_{0\leq t\leq T}\|h^{i-1}(t)\|_{L^{\infty}}+{C_{N,n,m}}
\left\|\f{h^{i-1}}{\langle v\rangle^{|\gamma|} w}\right\|_{L^2([0,T],L^2)}.\nonumber
\end{align}
Combining this with \eqref{3.1.23} and \eqref{3.1.24}, we get, for $t\in [0,T]$, that
\begin{align}\label{3.1.29}
|h^{i+1}(t,x,v)|\leq &Ce^{-c nT}\sup_{0\leq t\leq T}\|h^{i+1}(t)\|_{L^{\infty}}+\eta\max_{0\leq l\leq k}\{\sup_{0\leq t\leq T}\|h^{i-l}(t)\|_{L^{\infty}}\}\nonumber\\
&+C{n^{5/4}}m^{\gamma-1}\{\sup_{0\leq t\leq T}\|\nu^{-1}wg(t)\|_{L^{\infty}}+\sup_{0\leq t\leq T}|wr(t)|_{L^{\infty}(\gamma_-)}\}\nonumber\\
&+
{C_{N,n,m}}
\sup_{0\leq l\leq k}\left\|\f{h^{i-l}}{\langle v \rangle^{|\gamma|} w}\right\|_{L^{2}([0,T];L^2)},
\end{align}
where we have denoted
$$
\eta:=C{n^{5/2}}\bigg\{m^{\gamma-1}e^{-c nT}+m^{3+\gamma}+e^{-c (nT)^{\alpha}}+\left(\f{1}{2}\right)^{\hat{C}_2(nT)^{\f54}}+\f{m^{2(\gamma-1)}}{N}\bigg\}.
$$
We now take
$$
m=\left(\f{1}{32C}\right)^{\f{1}{3+\gamma}}{n^{-\f{5}{2(3+\gamma)}}},
$$
choose $n$ suitably large, and then choose $N$ large enough, so that it holds that
$$
{C}e^{-c nT}\leq \f12,\quad
\eta\leq \f{1}{16}.
$$
Then we obtain \eqref{3.1.4} from \eqref{3.1.29}. Finally, \eqref{3.1.5} directly follows from \eqref{3.1.4}. Therefore, the proof of Proposition \ref{prop3.1} is complete.
\end{proof}

\subsection{Approximation solutions}
It is very delicate to make the construction of approximation solutions. For readers' convenience, we first outline the procedure by four steps as follows.

\medskip
\noindent\underline{Step 1}. Construct the solution $f^{{j}
,\vep}$ to the following time-periodic problem:
\begin{equation}\label{3.2.1}\left\{
\begin{aligned}
&\pa_tf^{{j}
,\vep}+v\cdot\nabla_x f^{{j},\vep}+(\vep+\nu(v))f^{{j},\vep}=g,\\
&f^{{j},\vep}(t,x,v)|_{\g_-}=(1-\f1{{j}})P_{\g}f^{{j},\vep}+r.
\end{aligned}\right.
\end{equation}

\medskip
\noindent\underline{Step 2}. Construct the solution $f^{\vep}$ to the following time-periodic problem:
\begin{equation}\label{3.2.2}\left\{
\begin{aligned}
&\pa_tf^{\vep}+v\cdot\nabla_x f^{\vep}+(\vep+\nu(v))f^{\vep}=g,\\
&f^{\vep}(t,x,v)|_{\g_-}=P_{\g}f^\vep+r{,}
\end{aligned}\right.
\end{equation}
by passing to the limit ${j}\rightarrow\infty$.

\medskip
\noindent\underline{Step 3}. Make the uniform-in-$\lambda$ a priori estimates on the solution $f^{\lambda,\vep}$ to the following time-periodic problem:
\begin{equation}\label{3.2.3}\left\{
\begin{aligned}
&\pa_tf^{\lambda,\vep}+v\cdot\nabla_x f^{\lambda,\vep}+(\vep+\nu(v))f^{\lambda,\vep}=\lambda Kf^{\lambda,\vep}+g,\\
&f^{\lambda,\vep}(t,x,v)|_{\g_-}=P_{\g}f^{\lambda,\vep}+r,
\end{aligned}\right.
\end{equation}
and bootstrap from $\lambda=0$ to $\lambda=1$. Then the solution $f^{\vep}$ to
\begin{equation}\label{3.2.4}\left\{
\begin{aligned}
&\pa_tf^{\vep}+v\cdot\nabla_x f^{\vep}+(\vep+\nu(v))f^{\vep}=Kf^{\vep}+g,\\
&f^{\vep}(t,x,v)|_{\g_-}=P_{\g}f^{\vep}+r,
\end{aligned}\right.
\end{equation}
is therefore constructed. We remark that the zero-mass condition \eqref{3.0.3} is not necessary up to the present step.

\medskip
\noindent\underline{Step 4}. Take the limit $\vep\rightarrow 0.$ Note that in the limit process, the artificial damping term guarantees that the following key zero-mass condition
\begin{align}\label{3.2.5}
\int_{\Omega}\int_{\mathbb{R}^3}\pa_tf^{\vep}(t,x,v)\sqrt{\mu(v)}\dd v\dd x=\int_{\Omega}\int_{\mathbb{R}^3}f^{\vep}(t,x,v)\sqrt{\mu(v)}\dd v\dd x=0,
\end{align}
holds true for any $t\in \mathbb{R}$. In fact, let
$$
\rho^{\vep}(t):=\int_{\Omega}\int_{\mathbb{R}^3}f^{\vep}(t,x,v)\sqrt{\mu(v)}\dd v\dd x.
$$
Taking the inner product of \eqref{3.2.4} with $\sqrt{\mu(v)}$ over $\Omega\times\mathbb{R}^3$ and using the zero-mass condition \eqref{3.0.3}, we get
$$
\f{\dd \rho^{\vep}}{\dd t}+\vep\rho^{\vep}=0.
$$
Since $\rho^{\vep}(t)$ is periodic in time, we then obtain $\rho^{\vep}(t)\equiv0$. 

\medskip
In what follows, we will proceed the proof along the way mentioned above. The first lemma is related to the issue stated in Step 1. For the choice of ${j}$ in the second line of \eqref{3.2.1}, one can fix ${j_0}>1$ to be large enough such that
$$
\frac18 \left(1-\frac2{{j}}+\frac{3}{2{j}^2}\right)^{-\frac{k+1}{2}}\leq \frac12
$$
holds true for any ${j}\geq {j_0}$, where $k\sim (nT)^{5/4}$ is defined in \eqref{def.kno}. Then we only consider ${j}\geq {j_0}$ in the problem \eqref{3.2.1}.

\begin{lemma}\label{lem3.2.1}
Let $-3<\gamma\leq 1$, $\v>0$, {$0\leq q<1/8$ }and $\beta>3$.  Assume that $g$ and $r$ are time-periodic functions with period $T>0$ and satisfy
$$
\sup_{0\leq t\leq T}\|\nu^{-1}wg(t)\|_{L^\infty}+\sup_{0\leq t\leq T}|wr(t)|_{L^\infty(\g_-)}<\infty.
$$
Then there exists a unique solution $f^{{j},\vep}$ to \eqref{3.2.1}, which is time-periodic with period $T$, and satisfies
\begin{multline}\label{3.2.6}
\sup_{0\leq t \leq T}\|wf^{{j},\vep}(t)\|_{L^\infty}+\sup_{0\leq t\leq T}|wf^{{j},\vep}(t)|_{L^\infty{(\g)}}\\
\leq C_{\v,{j}}\Big( \sup_{0\leq t\leq T}|wr(t)|_{L^\infty(\gamma-)}+\sup_{0\leq t\leq T}\|\nu^{-1}wg(t)\|_{L^\infty} \Big),
\end{multline}
where the positive constant $C_{\v,{j}}>0$ depends only on $\v$ and ${j}$. Moreover, if  the domain $\Omega$ is convex, $g$ is continuous in 
$\mathbb{R}\times\Omega\times \mathbb{R}^3$, and $r$ is continuous {in $\mathbb{R}\times\g_-$,} 
then the solution $f^{{j},\vep}(t,x,v)$ is also continuous away from the grazing set $\mathbb{R}\times\gamma_0$.
\end{lemma}

\begin{proof}
For given $\eps>0$ and ${j}\geq {j_0}$, we shall construct the solution to \eqref{3.2.1}. To do so, we consider the  approximation sequence $\{f^i(t,x,v)\}_{i=0}^\infty$ iteratively solved by
\begin{equation}\label{3.2.7}
\left\{\begin{aligned}
&\pa_tf^{i+1}+v\cdot\nabla_x f^{i+1}+(\vep+\nu(v))f^{i+1}=g,\\
&f^{i+1}(t,x,v)|_{\g_-}=(1-\f1{{j}})P_{\g}f^{i}+r,
\end{aligned}\right.
\end{equation}
with  $f^0\equiv0$.
Here we have dropped $\vep$ and ${j}$ for brevity. Indeed, the solution to \eqref{3.2.7} can be constructed by the method of characteristics. 
Let
$$
h^{i+1}(t,x,v)=w(v)f^{i+1}(t,x,v).
$$
Then for any $t\in \mathbb{R}$ and {almost every} $(x,v)\in \b{\Omega}\times \mathbb{R}^3\setminus(\gamma_0\cup\gamma_-)$, one can write
\begin{align}\label{3.2.8}
h^{i+1}(t,x,v)=&e^{-(\vep+\nu(v))t_{{\mathbf{b}}}(x,v)} w(v)\left[(1-\f1{{j}})P_{\gamma}f^i+r\right](t-t_{{\mathbf{b}}}(x,v),x_{{\mathbf{b}}}(x,v),v)\nonumber\\
&+\int_{t-t_{{\mathbf{b}}}(x,v)}^te^{-(\vep+\nu(v))(t-s)}wg(s,x-(t-s)v,v)\dd s.
\end{align}
Note that for $(x,v)\in \gamma_- $, 
it is direct to write
\begin{align}\label{3.2.8-1}
h^{i+1}(t,x,v)=w(v)\left[(1-\f1{{j}})P_{\gamma}f^i+r\right](t,x,v).
\end{align}
Now we use the induction argument to show that
\begin{equation}\label{3.2.9}
h^{i}(t,x,v) \text{ is time-periodic with period }T>0,
\end{equation}
and the following estimate holds true:
\begin{multline}\label{3.2.10}
\sup_{0\leq t\leq T}\|h^{i}(t)\|_{L^{\infty}}+\sup_{0\leq t\leq T}|h^{i}(t)|_{L^{\infty}{(\g)}}\\
\leq C_{{j},\vep,i}\left(
\sup_{0\leq t\leq T}\|\nu^{-1}wg(t)\|_{L^{\infty}}+\sup_{0\leq t\leq T}|wr(t)|_{L^{\infty}(\gamma_-)}\right).
\end{multline}
Indeed, for $i=0$, it is obvious 
to see that \eqref{3.2.9} and \eqref{3.2.10} are satisfied. 
Assume that \eqref{3.2.9} and \eqref{3.2.10} hold for $i\geq 0$.  
\eqref{3.2.8} implies that
\begin{align}\label{3.2.11}
&h^{i+1}(t+T,x,v)\notag\\
&=e^{-(\vep+\nu(v))t_{{\mathbf{b}}}} w(v)\left[(1-\f1{{j}})P_{\gamma}f^i+r\right](t+T-t_{{\mathbf{b}}},x_{{\mathbf{b}}},v)\nonumber\\
&\quad +\int_{t+T-t_{{\mathbf{b}}}}^{t+T}e^{-(\vep+\nu(v))(t+T-s)}wg(s,x-(t+T-s)v,v)\dd s.
\end{align}
Note that by the induction assumption that both $f^i$ and $r$ are time-periodic functions with period $T$, the first term on the right-hand side of \eqref{3.2.11} is equal to
$$
e^{-(\vep+\nu(v))t_{{\mathbf{b}}}} w(v)\left[(1-\f1{{j}})P_{\gamma}f^i+r\right](t-t_{{\mathbf{b}}},x_{{\mathbf{b}}},v).$$
For the second term, taking change of variables $s\rightarrow s-T $, we get that
\begin{align}
&\int_{t+T-t_{{\mathbf{b}}}}^{t+T}e^{-(\vep+\nu(v))(t+T-s)}wg(s,x-(t+T-s)v,v)\dd s\nonumber\\
&=\int_{t-t_{{\mathbf{b}}}}^{t}e^{-(\vep+\nu(v))(t-s)}wg(s+T,x-(t-s)v,v)\dd s\nonumber\\
&=\int_{t-t_{{\mathbf{b}}}}^{t}e^{-(\vep+\nu(v))(t-s)}wg(s,x-(t-s)v,v)\dd s,\nonumber
\end{align}
where in the last line we have used the fact that $g$ is periodic in time with period $T$. Therefore, it follows from \eqref{3.2.11} that
$$
h^{i+1}(t+T,x,v)\equiv h^{i+1}(t,x,v),
$$
so, \eqref{3.2.9} holds true for $i+1$. Moreover, to show \eqref{3.2.10} for $i+1$, it follows from \eqref{3.2.8} that
\begin{align}
&\sup_{0\leq t\leq T}\{\|h^{i+1}(t)\|_{L^{\infty}}+|h^{i+1}(t)|_{L^{\infty}(\gamma_+)}\}\nonumber\\
&\leq C\sup_{0\leq t\leq T}\{|h^{i}(t)|_{L^{\infty}(\gamma_+)}+|wr(t)|_{L^{\infty}(\gamma_-)}+\|\nu^{-1}wg(t)\|_{L^{\infty}}\}\nonumber\\
&\leq C_{{j},i}\sup_{0\leq t\leq T}\{|wr(t)|_{L^{\infty}(\gamma_-)}+\|\nu^{-1}wg(t)\|_{L^{\infty}}\},\nonumber
\end{align}
and also one obtains by \eqref{3.2.8-1} that
\begin{align}
\sup_{0\leq t\leq T}|h^{i+1}(t)|_{L^{\infty}(\g_-)}&\leq C\sup_{0\leq t\leq T}|h^{i}(t)|_{L^{\infty}(\gamma_+)}+C\sup_{0\leq t\leq T}|wr(t)|_{L^{\infty}(\gamma_-)}\nonumber\\
&\leq C_{{j},i}\sup_{0\leq t\leq T}\{|wr(t)|_{L^{\infty}(\gamma_-)}+\|\nu^{-1}wg(t)\|_{L^{\infty}}\}.\nonumber
\end{align}
Combing the above two estimates gives the proof of \eqref{3.2.10} for $i+1$. Therefore, by induction \eqref{3.2.9} and \eqref{3.2.10} are satisfied for all $i$.
Then, each $h^{i}(t,x,v)$ is well-defined in $L^{\infty}$ and time-periodic with period $T>0$. Moreover, if $\Omega$ is convex, $t_{{\mathbf{b}}}(x,v)$ and $x_{{\mathbf{b}}}(x,v)$ are smooth away from $\g_0$. If $g$ and $r$ are further continuous, then each $f^{i}(t,x,v)$ is also continuous for away from the grazing set $\mathbb{R}\times\g_0$. 

Next, we need to obtain the uniform-in-$i$ estimate on the solution sequence $f^{i}$. We first treat it in the $L^2$ setting. Taking the inner product of \eqref{3.2.7} with $f^{i+1}$ over $[0,T]\times\Omega\times \mathbb{R}^3$ and using the periodicity of $f^{i+1}$, we obtain that
\begin{align}\label{3.2.12}
&\f{1}{2}\int_0^T|f^{i+1}(s)|_{L^2(\gamma_+)}^2\,\dd s+\int_0^T\vep\|f^{i+1}(s)\|_{L^2}^2+\f34\|\nu^{1/2}f^{i+1}(s)\|_{L^{2}}^2\dd s\nonumber
\nonumber\\
&\leq \f{1}{2}(1-\f2{{j}}+\f{3}{2{j}^2})\int_0^T|f^i(s)|_{L^2(\gamma_+)}^2\dd s\notag\\
&\quad+\int_0^T\|\nu^{-1/2}g(s)\|_{L^2}^2+C_{{j}}|r(s)|_{L^{2}(\g_-)}^2\dd s,
\end{align}
where we have used the fact that
$
|P_{\g}f^i|_{L^{2}(\g_-)}=|P_{\g}f^i|_{L^2(\g_+)}\leq |f^i|_{L^{2}(\g_+)}.
$
For the difference $f^{i+1}-f^i$, 
in a similar way we have
\begin{align}\label{add.iter}
&\f{1}{2}\int_0^T|[f^{i+1}-f^i](s)|_{L^2(\gamma_+)}^2\,\dd s\notag\\
&\quad +\int_0^T\vep\|[f^{i+1}-f^i](s)\|_{L^2}^2+\f34\|\nu^{1/2}[f^{i+1}-f^i](s)\|_{L^{2}}^2\dd s\nonumber
\nonumber\\
&\leq \f{1}{2}(1-\f2{{j}}+\f{3}{2{{j}}^2})\int_0^T|[f^i-f^{i-1}](s)|_{L^2(\gamma_+)}^2\dd s, 
\end{align}
and hence, by iteration the right-hand side of \eqref{add.iter} is further bounded by
\begin{multline}\label{3.2.13}
\f{1}{2}(1-\f2{{j}}+\f{3}{2{{j}}^2})^i\int_0^T|[f^1-f^{0}](s)|_{L^2(\gamma_+)}^2\dd s\\
\leq \f{1}{2}(1-\f2{{j}}+\f{3}{2{{j}}^2})^i\cdot\bigg\{C_{{j}}\int_0^T|r(s)|^2_{L^{2}(\g_-)}+\|\nu^{-1/2}g(s)\|^2_{L^2}\dd s\bigg\},
\end{multline}
where in the second line we have used \eqref{3.2.12} for $i=0$ as well as $f^0\equiv 0$.  
As ${j}_0>1$ is chosen to be large enough, one has
$
0<1-\f2{{j}}+\f{3}{2{{j}}^2}<1
$
for any ${j}\geq {j}_0$. It then follows from \eqref{add.iter} and \eqref{3.2.13} that
$\{f^i\}_{i=0}^{\infty}$ is  a Cauchy sequence in $L^2$. Moreover,
for any $i\geq 0$, it holds that
\begin{equation*}
\int_0^T\|\nu^{1/2}f^{i}(s)\|_{L^2}^2+|f^i(s)|_{L^2(\gamma_+)}^2\dd s\leq C_{{j}} \int_0^T|r(s)|^2_{L^{2}(
{\g_-})}+\|\nu^{-1/2}g(s)\|^2_{L^2}\dd s,
\end{equation*}
and hence
the following uniform-in-$i$ estimate holds true:
\begin{multline}\label{3.2.14}
\int_0^T\|\nu^{1/2}f^{i}(s)\|_{L^2}^2+|f^i(s)|_{L^2(\gamma_+)}^2\dd s\\
\leq C_{{j}}\big\{\sup_{0\leq t\leq T}|wr(t)|_{L^{\infty}(\g_-)}+\sup_{0\leq t\leq T}\|\nu^{-1}wg(t)\|_{L^{\infty}}\big\}^2.
\end{multline}

Next we turn to treat the uniform  estimate in the $L^{\infty}$ setting in terms of the results obtained in the previous subsection. Note that 
Proposition \ref{prop3.1} is also valid if the boundary condition of the problem \eqref{3.1.1} is replaced by
\begin{equation*}
h^{i+1}(t,x,v)|_{\gamma_-}=\f{1-\frac{1}{{j}}}{\tilde{w}(v)} \int_{v'\cdot n(x)>0} h^i({t},x,v') \tilde{w}(v') \dd\sigma'+w(v)r({t},x,v),
\end{equation*}
namely, we have only changed $1$ to $1-1/{j}$.
Correspondingly one can deduce the mild formulation \eqref{3.1.2}, and prove Lemma \ref{lem.smbd} and Proposition \ref{prop3.1}. Particularly,  all constants in \eqref{3.1.4} and \eqref{3.1.5} are independent of ${j}$. Then, using \eqref{3.1.4},  we obtain that
\begin{align}
&\sup_{0\leq t\leq T}\|h^{i+1}(t)\|_{L^\infty}+\sup_{0\leq t\leq T}|h^{i+1}(t)|_{L^\infty(\g)}\nonumber\\
&\quad\leq \frac18 \max_{0\leq l\leq k} \{\sup_{0\leq t\leq T}\|h^{i-l}(t)\|_{L^\infty}\}+C\sup_{0\leq t\leq T}\Big\{ |wr(t)|_{L^\infty(\gamma_-)}+\|\nu^{-1}wg(t)\|_{L^\infty}\Big\}\nonumber\\
&\qquad+C\sup_{0\leq l\leq k} \{ \|\nu^{1/2}f^{i-l}\|_{L^2([0,T];L^2)}\}.\nonumber
\end{align}
It then follows from \eqref{3.2.14} that
\begin{align}\label{3.2.15}
&\sup_{0\leq t\leq T}\|h^{i+1}(t)\|_{L^\infty}{+\sup_{0\leq t\leq T}|h^{i+1}(t)|_{L^\infty(\g)}}\nonumber\\
&\quad\leq \frac18 \max_{0\leq l\leq k} \{\sup_{0\leq t\leq T}\|h^{i-l}(t)\|_{L^\infty}\}
+C_{{j}}\sup_{0\leq t\leq T}\big\{|wr(t)|_{L^{\infty}(\g_-)}+\|\nu^{-1}wg(t)\|_{L^{\infty}}\big\}.
\end{align}
Applying \eqref{A.1} to \eqref{3.2.15}, it holds that for $i\geq k+1$,
\begin{align}\label{3.2.16}
&\sup_{0\leq t\leq T}\|h^{i}(t)\|_{L^{\infty}}{+\sup_{0\leq t\leq T}|h^{i}(t)|_{L^\infty(\g)}}\nonumber\\
&\quad\leq \f18\max_{1\leq l\leq 2k}\{\sup_{0\leq t\leq T}\|h^{l}(t)\|_{L^{\infty}}\}\nonumber\\
&\qquad+\f{8+k}{7}C_{{j}}\big\{\sup_{0\leq t\leq T}|wr(t)|_{L^{\infty}(\g_-)}+\sup_{0\leq t\leq T}\|\nu^{-1}wg(t)\|_{L^{\infty}}\big\}\nonumber\\
&\quad\leq C_{{j}}\big\{\sup_{0\leq t\leq T}|wr(t)|_{L^{\infty}(\g_-)}+\sup_{0\leq t\leq T}\|\nu^{-1}wg(t)\|_{L^{\infty}}\big\},
\end{align}
where we have used \eqref{3.2.10} for $i=1,\cdots,2k$ in the last inequality. Combining \eqref{3.2.16} with \eqref{3.2.10}, we obtain that for $i\geq 1$,
\begin{align}\label{3.2.17}
&\sup_{0\leq t\leq T}\|h^{i}(t)\|_{L^{\infty}}{+\sup_{0\leq t\leq T}|h^{i}(t)|_{L^\infty(\g)}}\nonumber\\
&\quad\leq C_{{j}}\big\{\sup_{0\leq t\leq T}|wr(t)|_{L^{\infty}(\g_-)}+\sup_{0\leq t\leq T}\|\nu^{-1}wg(t)\|_{L^{\infty}}\big\}.
\end{align}
Similarly for obtaining \eqref{3.2.17}, one can apply \eqref{3.1.4} to $h^{i+2}-h^{i+1}$ to get
\begin{align}\label{3.2.18}
&\sup_{0\leq t\leq T}\|[h^{i+2}-h^{i+1}](t)\|_{L^{\infty}}+\sup_{0\leq t\leq T}|[h^{i+2}-h^{i+1}](t)|_{L^{\infty}{(\g)}}\nonumber\\
&\leq \f18\max_{0\leq l\leq k}\{\sup_{0\leq t\leq T}\|[h^{i+1-l}-h^{i-l}](t)\|_{L^{\infty}}\}\notag\\
&\qquad+C\max_{0\leq l\leq k}\{\|\nu^{1/2}[f^{i+1-l}-f^{i-l}]\|_{L^{2}([0,T];L^2)}\}\nonumber\\
&\leq \f18\max_{0\leq l\leq k}\{\sup_{0\leq t\leq T}\|[h^{i+1-l}-h^{i-l}](t)\|_{L^{\infty}}\}\notag\\
&\qquad+C_{{j}}\eta_{{j}}^{i-k}
{\left\{\int_0^T|r(s)|^2_{L^{2}({\g_-}
)}+\|\nu^{-1/2}g(s)\|^2_{L^2}\dd s\right\}^{1/2}}\nonumber\\
&\leq \f18\max_{0\leq l\leq k}\{\sup_{0\leq t\leq T}\|[h^{i+1-l}-h^{i-l}](t)\|_{L^{\infty}}\}\notag\\
&\qquad+
C_{{j}}\sup_{0\leq t\leq T}\{|wr(t)|_{L^{\infty}(\g_-)}+\|\nu^{-1}wg(t)\|_{L^{\infty}}\}\eta_{{j}}^{i+k+1},
\end{align}
where we have denoted
$
\eta_{{j}}:=\sqrt{1-\f2{{j}}+\f3{{j}^2}}.
$
Let ${j}_0>1$ be suitably large such that
$
\f18\eta_{{j}}^{-k-1}\leq\f12
$
for any ${j}\geq {j}_0$. Then, applying \eqref{A.1-1} to \eqref{3.2.18}, we obtain that for $i\geq k+1$,
\begin{align}\label{3.2.19}
&\sup_{0\leq t\leq T}\|[h^{i+2}-h^{i+1}](t)\|_{L^{\infty}}+\sup_{0\leq t\leq T}|[h^{i+2}-h^{i+1}](t)|_{L^{\infty}{(\g)}}\nonumber\\
&\leq \left(\f18\right)^{[\f{i}{k+1}]}\max_{0\leq l\leq 2k+1}\{\|h^{l}(t)\|_{L^{\infty}}\}\notag\\
&\qquad+C_{{j}}\sup_{0\leq t\leq T}\{|wr(t)|_{L^{\infty}(\g_-)}+\|v^{-1}wg(t)\|_{L^{\infty}}\}\cdot\eta_{{j}}^i\nonumber\\
&\leq C_{{j}}\big\{\left(\f18\right)^{[\f{i}{k+1}]}+\eta_{{j}}^i\big\}\sup_{0\leq t\leq T}\{|wr(t)|_{L^{\infty}(\g_-)}+\|v^{-1}wg(t)\|_{L^{\infty}}\}.
\end{align}
Hence, from \eqref{3.2.19}, we see that  $\{h^{i}\}$ is also a Cauchy sequence in $L^{\infty}$. Let $h(t,x,v)$ be the limit function of $h^{i}$ in $L^\infty$. It is straightforward to check that $f:=\f{h}{w}$ solves \eqref{3.2.1} for ${j}\geq {j}_0$. Furthermore, since each $f^i$ is a time-periodic function with period $T$ and $h^i=wf^i$ converges to $h$ in $L^{\infty}$, then $f=\f{h}{w}$ is also  periodic in time with the same period $T$. If $\Omega$ is convex, the continuity of $f$ directly follows from the continuity of $f^i$. Moreover, taking the limit $i \rightarrow \infty$ in \eqref{3.2.14}, we get that
\begin{align}\label{3.2.20}
\|\nu^{1/2}f\|_{L^{2}([0,T];L^2)}\leq C_{{j}}\sup_{0\leq t\leq T}\{|wr(t)|_{L^{\infty}(\g_-)}+\|\nu^{-1}wg(t)\|_{L^{\infty}}\}.
\end{align}
Then the $L^{\infty}$ bound \eqref{3.2.6} directly follows from \eqref{3.1.5} and \eqref{3.2.20}. The proof of Lemma \ref{lem3.2.1} is therefore complete.
\end{proof}

As mentioned before, Lemma \ref{lem3.2.1} is the first step for obtaining the approximation solutions $f^{{j},\varepsilon}$ to \eqref{3.2.1}. We now turn to the second step to establish the solvability of the problem \eqref{3.2.2} by letting ${j}\to \infty$. For the time being, in the following lemma we omit the dependence of $f^{{j},\varepsilon}$ on $\varepsilon$ for brevity.

\begin{lemma}\label{lem3.2.2}
Let $-3<\gamma\leq 1$, $\v>0$, {$0\leq q<1/8$} and $\beta>3$. Under the same assumption as in Lemma \ref{lem3.2.1}, there exists a unique time-periodic solution $f(t,x,v)$ to \eqref{3.2.2} satisfying the estimate
\begin{multline}\label{3.2.2-2}
\sup_{0\leq t\leq T}\big\{\|wf(t)\|_{L^\infty} +|wf(t)|_{L^\infty{(\g)}}\big\} \\
\leq C \sup_{0\leq t\leq T}\big\{ |wr(t)|_{L^\infty(\g_-)}+\|\nu^{-1}wg(t)\|_{L^\infty} \big\}.
\end{multline}
Furthermore, if $\Omega$ is convex,  $g$ is continuous in
$\mathbb{R}\times\Omega\times\mathbb{R}^3$ and
$r$ is continuous {in $\mathbb{R}\times \g_-$,} 
 then $f(t,x,v)$ is also continuous away from the grazing set $\mathbb{R}\times\gamma_0$.
\end{lemma}
\begin{proof}
We shall first obtain the uniform-in-${j}$ estimate on the solutions $f^{{j}}$ to \eqref{3.2.1} and then show that $h^{{j}}:=wf^{{j}}$ is Cauchy in $L^{\infty}$.

To treat $L^\infty$ estimates, we should start from $L^2$ estimates. Taking the inner product of \eqref{3.2.1} with $f^{{j}}$ over $[0,T]\times \Omega\times \mathbb{R}^3$ gives that
\begin{align}
&\int_0^T\vep\|f^{{j}}(s)\|_{L^2}^2+\f12\|\nu^{1/2}f^{{j}}(s)\|_{L^2}^2+\f12|f^{{j}}(s)|_{L^2(\g_+)}^2\dd s\nonumber\\
&\leq C\int_0^T\|\nu^{-1/2}g(s)\|_{L^2}^2\dd s+\f12\int_0^T\big|(1-\f1{{j}})P_{\g}f^{{j}}+r\big|_{L^2(\g_-)}^2\dd s\nonumber\\
&\leq C\int_0^T\|\nu^{-1/2}g(s)\|_{L^2}^2\dd s+\f{1+\eta}{2}\int_0^T|P_\g f^{{j}}(s)|_{L^{2}(\g_+)}^2\dd s+C_\eta\int_0^T|r(s)|^2_{L^{2}(\g_-)}\dd s,\nonumber
\end{align}
which further implies that
\begin{align}\label{3.2.21}
&\int_0^T\vep\|f^{{j}}(s)\|_{L^2}^2+\f12\|\nu^{1/2}f^{{j}}(s)\|_{L^2}^2+\f12|(I-P_\gamma)f^{{j}}(s)|_{L^2(\g_+)}^2\dd s\nonumber\\
&\leq C\int_0^T\|\nu^{-1/2}g(s)\|_{L^2}^2\dd s+\f{\eta}{2}\int_0^T|P_\g f^{{j}}(s)|_{L^{2}(\g_+)}^2\dd s+C_\eta\int_0^T|r(s)|^2_{L^{2}(\g_-)}\dd s,
\end{align}
where $\eta>0$ can be arbitrarily small. To estimate the second term on the right-hand side of \eqref{3.2.21}, using the same idea as in \cite{EGKM}, we {recall }
 the near-grazing set {$\g_+^{\vep'}$ defined in \eqref{cut}} 
{and split $P_{\g}f^j=P_{\g}(f^j\Fi_{\g^{\vep'}})+P_{\g}(f^j\Fi_{\g_+\setminus\g^{\vep'}_+}).$ By a direct computation, we have
\begin{align}
|P_{\g}(f^j\Fi_{\g^{\vep'}})|_{L^2(\g_-)}\leq C\vep'|f^j|_{L^2(\g_+)}\leq C\vep'|P_{\g}f^j|_{L^2(\g_+)}+C\vep'|(I-P_{\g})f^j|_{L^2(\g_+)},\nonumber
\end{align}
and
\begin{align}
|P_{\g}(f^j\Fi_{\g_+\setminus\g^{\vep'}_+})|_{L^2(\g_-)}^2=&\int_{\g_-}\mu(v)|n(x)\cdot v|\dd \gamma\nonumber\\
&\times\left(\int_{n(x)\cdot v'>0}e^{-\f{|v|^2}{8}}f^j\Fi_{\g_+\setminus\g^{\vep'}_+}e^{\f{|v|^2}{8}}\sqrt{\mu(v)}|n(x)\cdot v'|\dd v'\right)^{2},\nonumber\\
\leq &C|e^{-\f{|v|^2}{8}}f^j\Fi_{\g_+\setminus\g^{\vep'}_+})|_{L^2(\g_+)}^2.\nonumber
\end{align}
From the first equation of \eqref{3.2.1}, we have
\begin{align}
(\pa_t+v\cdot \nabla_x )e^{-\f{1}{4}|v|^2}(f^{{j}})^2=2e^{-\f{1}{4}|v|^2}gf^{{j}}-2[\vep+ \nu(v)] e^{-\f{1}{4}|v|^2} (f^{{j}})^2,\nonumber
\end{align}
which implies that
$$
\|(\pa_t+v\cdot \nabla_x )e^{-\f{1}{4}|v|^2}(f^{{j}})^2\|_{L^1}\leq C\|e^{-\f{|v|^2}{16}}f^{{j}}\|_{L^2}^2+C\|e^{-\f{|v|^2}{16}}g\|_{L^2}^2.
$$
Thus, from the trace Lemma \ref{lemT}, it follows that
\begin{align}
&\int_0^T|e^{-\f{1}{8}|v|^2}f^{{j}}(s)\Fi_{\g_+\setminus\g_+^{\vep'}}|_{L^2(\g_+)}^2\dd s\nonumber\\
&=\int_0^T\left|[e^{-\f{1}{8}|v|^2}f^{{j}}]^2(s)\Fi_{\g_+\setminus\g_+^{\vep'}}\right|_{L^1(\g_+)}\dd s
\nonumber\\
&\lesssim_{\vep',\Omega}\int_0^T\|(\pa_t+v\cdot \nabla_x )e^{-\f{1}{4}|v|^2}(f^{{j}})^2\|_{L^1}+\|e^{-\f14|v|^2}(f^{{j}})^2(s)\|_{L^1}\dd s\notag\\
&\qquad+\|e^{-\f14|v|^2}(f^{{j}})^2(0)\|_{L^1}\nonumber\\
&\lesssim_{\vep',\Omega}\int_0^T\|e^{-\f{|v|^2}{16}}f^{{j}}(s)\|_{L^2}^2\dd s+\int_0^T\|e^{-\f{|v|^2}{16}}g(s)\|_{L^2}^2\dd s+\sup_{0\leq t\leq T}\|e^{-\f{|v|^2}{16}}f^{{j}}(t)\|_{L^{\infty}}^2.\nonumber
\end{align}
Collecting these estimates, we have
\begin{align}\label{3.2.22}
&\int_0^T|P_{\g}f^{{j}}(s)|_{L^2(\g_+)}^2\dd s\nonumber\\
&\quad\leq C\vep'\int_0^T|P_{\g}f^{{j}}(s)|_{L^2(\g_+)}^2\dd s+C\vep'\int_0^T|(I-P_{\g})f^{{j}}(s)|_{L^2(\g_+)}^2\dd s\nonumber\\
&\qquad+C\int_0^T|e^{-\f{1}{8}|v|^2}P_\g f^{{j}}\Fi_{\g_+\setminus\g_+^{\vep'}}(s)|_{L^2(\g_+)}^2\dd s\nonumber\\
&\quad\leq C\vep'\int_0^T|P_{\g}f^{{j}}(s)|_{L^2(\g_+)}^2\dd s+C\vep'\int_0^T|(I-P_{\g})f^{{j}}(s)|_{L^2(\g_+)}^2\dd s\nonumber\\
&\qquad+C_{\vep'}\int_0^T\|e^{-\f{|v|^2}{16}}f^{{j}}(s)\|_{L^2}^2+\|e^{-\f{|v|^2}{16}}g(s)\|_{L^2}^2\dd s+C_{\vep'}\sup_{0\leq t\leq T}\|e^{-\f{|v|^2}{16}}f^{{j}}(t)\|_{L^{\infty}}^2\nonumber\\
&\quad\leq C\int_0^T|(I-P_{\g})f^{{j}}(s)|_{L^2(\g_+)}^2+\|e^{-\f{|v|^2}{16}}f^{{j}}(s)\|_{L^2}^2+\|e^{-\f{|v|^2}{16}}g(s)\|_{L^2}^2\dd s\nonumber\\
&\qquad+C\sup_{0\leq t\leq T}\|e^{-\f{|v|^2}{16}}f^{{j}}(t)\|_{L^{\infty}}^2.
\end{align}
Here we have taken $\vep'>0$ suitably small. Plugging \eqref{3.2.22} back to \eqref{3.2.21}, }
we get that
\begin{align}\label{3.2.24}
&\int_0^T\vep\|f^{{j}}(s)\|_{L^2}^2+\|\nu^{1/2}f^{{j}}(s)\|_{L^2}^2+|(I-P_\gamma)f^{{j}}(s)|_{L^2(\g_+)}^2\dd s\nonumber\\
&\leq {C}\eta\int_0^T\|\nu^{1/2}f^{{j}}(s)\|_{L^2}^2+|(I-P_\gamma)f^{{j}}(s)|_{L^2(\g_+)}^2\dd s
+{C}\eta \sup_{0\leq t\leq T}\|{e^{-\f{|v|^2}{16}}}f^{{j}}(t)\|_{L^{\infty}}^2\nonumber\\
&\quad+{C}\int_0^T\|\nu^{-1/2}g(s)\|_{L^2}^2\dd s+{C}\int_0^T|r(s)|^2_{L^{2}(\g_-)}\dd s.
\end{align}
Then, for any $\eta$ with $0<\eta\leq \eta_1:=\f{1}{2C}$, it follows from \eqref{3.2.24} that
\begin{align}\label{3.2.25}
&\int_0^T\vep\|f^{{j}}(s)\|_{L^2}^2+\f12\|\nu^{1/2}f^{{j}}(s)\|_{L^2}^2+\f12|(I-P_\gamma)f^{{j}}(s)|_{L^2(\g_+)}^2\dd s\nonumber\\
&\leq
C\eta\sup_{0\leq t\leq T}\|{e^{-\f{|v|^2}{16}}}f^{{j}}(t)\|_{L^{\infty}}^2+C_{\eta}\int_0^T\|\nu^{-1/2}g(s)\|_{L^2}^2\dd s+C_{\eta}\int_0^T|r(s)|^2_{L^{2}(\g_-)}\dd s\nonumber\\
&\leq C\eta\sup_{0\leq t\leq T}\|{e^{-\f{|v|^2}{16}}}f^{{j}}(t)\|_{L^{\infty}}^2+C_{\eta}\sup_{0\leq t\leq T}\{\|\nu^{-1}wg(t)\|_{L^\infty}+|wr(t)|_{L^\infty(\g_-)}\}^2.
\end{align}
On the other hand, by applying the $L^{\infty}$ estimate \eqref{3.1.5} to $h^{{j}}:=wf^{{j}}$, one has
\begin{multline*}
\sup_{0\leq t\leq T}\{\|h^{{j}}(t)\|_{L^{\infty}}+|h^{{j}}(t)|_{L^{\infty}{(\g)}}\}\\
\leq C\sup_{0\leq t\leq T}\{\|\nu^{-1}wg(t)\|_{L^\infty}+|wr(t)|_{L^\infty(\g_-)}\}
+C\|\nu^{1/2}f^{{j}}\|_{L^{2}([0,T];L^2)}.
\end{multline*}
Plugging \eqref{3.2.25} in the above estimate gives
\begin{multline*}
\sup_{0\leq t\leq T}\{\|h^{{j}}(t)\|_{L^{\infty}}+|h^{{j}}(t)|_{L^{\infty}{(\g)}}\}\\
\leq C{\eta^{1/2}}\sup_{0\leq t\leq T}\|h^{{j}}(t)\|_{L^{\infty}}+C_{\eta}\sup_{0\leq t\leq T}\{\|\nu^{-1}wg(t)\|_{L^\infty}+|wr(t)|_{L^\infty(\g_-)}\}.
\end{multline*}
Further letting $\eta>0$ be small enough, it then follows that
\begin{equation}\label{3.2.26}
\sup_{0\leq t\leq T}\{\|h^{{j}}(t)\|_{L^{\infty}}+|h^{{j}}(t)|_{L^{\infty}{(\g)}}\}
\leq C\sup_{0\leq t\leq T}\{\|\nu^{-1}wg(t)\|_{L^\infty}+|wr(t)|_{L^\infty(\g_-)}\}.
\end{equation}
This completes the uniform-in-${j}$ $L^\infty$ estimates.

Next, we need to show that $h^{{j}}:=wf^{{j}}$ is Cauchy in $L^{\infty}$.  For this, we consider the difference $h^{{{j}}_2}-h^{{j}_1}$. Note that $f^{{j}_2}-f^{{j}_1}=w^{-1}\big(h^{{j}_2}-h^{{j}_1}\big)$ solves
\begin{equation}
\left\{\begin{aligned}
&\pa_t(f^{{j}_2}-f^{{{j}}_1})+v\cdot\nabla_x(f^{{j}_2}-f^{{j}_1})+[\vep+\nu(v)](f^{{j}_2}-f^{{j}_1})=0,\\
&(f^{{j}_2}-f^{{j}_1})|_{\gamma_-}=(1-\f{1}{{j}_2})P_{\g}(f^{{j}_2}-f^{{j}_1})+(\f1{{j}_1}-\f1{{j}_2})P_{\g}f^{{j}_1}.\nonumber
\end{aligned}\right.
\end{equation}
Then,  by similar energy estimates made above, it holds that
\begin{align}
&\int_0^T\|\nu^{1/2}(f^{{j}_2}-f^{{j}_1})(s)\|_{L^2}^2\dd s\nonumber\\
 &\leq \eta\sup_{0\leq t\leq T}\|(h^{{j}_2}-h^{{j}_1})(t)\|_{L^{\infty}}^2+C_{\eta}\int_0^T|(\f1{{j}_2}-\f1{{j}_1})P_{\g}f^{{j}_1}(s)|_{L^{2}(\g_-)}^2\dd s\nonumber\\
 &\leq \eta\sup_{0\leq t\leq T}\|(h^{{j}_2}-h^{{j}_1})(t)\|_{L^{\infty}}^2\notag\\
&\qquad+C_{\eta}\bigg(\f{1}{{j}_1^2}+\f{1}{{j}_2^2}\bigg)\sup_{0\leq t\leq T}\{\|\nu^{-1}wg(t)\|_{L^\infty}+|wr(t)|_{L^\infty(\g_-)}\}^{{2}},\nonumber
\end{align}
where we have used \eqref{3.2.26} in the last inequality. Again, applying {\eqref{3.1.5}} 
 to the difference $h^{{j}_2}-h^{{j}_1}$, we get that
\begin{align}
&\sup_{0\leq t\leq T}\|(h^{{j}_2}-h^{{j}_1})(t)\|_{L^{\infty}}+\sup_{0\leq t\leq T}|(h^{{j}_2}-h^{{j}_1})(t)|_{L^{\infty}{(\g)}}\nonumber\\
&\leq C\sup_{0\leq t\leq T}\bigg|w(\f1{{j}_2}-\f1{{j}_1})P_{\g}f^{{j}_1}\bigg|_{L^{\infty}(\g_-)}+C\|\nu^{1/2}(f^{{j}_2}-f^{{j}_1})\|_{L^2([0,T];L^2)}\nonumber\\
&\leq C{\eta^{1/2}}\sup_{0\leq t\leq T}\|(h^{{j}_2}-h^{{j}_1})(t)\|_{L^{\infty}}\notag\\
&\qquad+C_{\eta}\bigg(\f{1}{{j_1}}+\f{1}{{j_2}}\bigg)\sup_{0\leq t\leq T}\{\|\nu^{-1}wg(t)\|_{L^\infty}+|wr(t)|_{L^\infty(\g_-)}\}.\nonumber
\end{align}
Taking $\eta>0$ suitably small, the above estimate yields that $h^{{j}}$ is Cauchy in $L^{\infty}$. Let $h(t,x,v)$ be the limit function of $h^{{j}}$. It is direct to check that $f:=\f{h}{w}$ solves \eqref{3.2.2}, and the estimate \eqref{3.2.2-2} follows from \eqref{3.2.26}.  Moreover, since  each $f^{{j}}$ is time-periodic with period $T$, then $f$ is also time-periodic with the same period $T$. The continuity follows in a similar way. Thus, the proof of Lemma \ref{lem3.2.2} is complete.
\end{proof}

We now move to the third step for treating the existence and uniform estimates of solutions to the linear problem \eqref{3.2.4} where the linear collision term is involved.  For the proof ,we follow the same strategy as in \cite{DHWZ}.

\begin{lemma}\label{lem3.2.3}
Let $-3<\gamma\leq 1$, $\vep>0$, {$0\leq q<1/8$} and $\beta>3$. Under the same assumption as in Lemma \ref{lem3.2.1}, the linear problem \eqref{3.2.4} admits a unique time-periodic solution $f^{\vep}(t,x,v)$ with period $T$, satisfying the following estimate:
\begin{multline}\label{3.2.27}
\sup_{0\leq t\leq T}\big\{\|wf^{\vep}(t)\|_{L^\infty} +|wf^{\vep}(t)|_{L^\infty{(\g)}}\big\}\\
 \leq C_{\vep}\sup_{0\leq t\leq T}\big\{ |wr(t)|_{L^\infty(\g_-)}+\|\nu^{-1}wg(t)\|_{L^\infty} \big\}.
\end{multline}
Moreover, if $\Omega$ is convex, $g$ is continuous in
$\mathbb{R}\times\Omega\times\mathbb{R}^3$, and  $r$ is continuous {in $\mathbb{R}\times\g_-$, }
then $f^\vep(t,x,v)$ is also continuous away from $\mathbb{R}\times\gamma_0$.
\end{lemma}
\begin{proof}
The proof relies on the following uniform-in-$\lambda$ estimate on the solution $f^{\lambda,\vep}$ to the modified linear problem \eqref{3.2.3} {for $0\leq \lambda\leq 1$}:
\begin{multline}\label{3.2.28}
\sup_{0\leq t\leq T}\big\{\|wf^{\lambda,\vep}(t)\|_{L^\infty} +|wf^{\lambda,\vep}(t)|_{L^\infty{(\g)}}\big\} \\
\leq C_{\vep} \sup_{0\leq t\leq T}\big\{ |wr(t)|_{L^\infty(\g_-)}+\|\nu^{-1}wg(t)\|_{L^\infty} \big\},
\end{multline}
where the positive constant $C_{\vep}$ is independent of $\la$ but may depend on $\vep$.  Once \eqref{3.2.28} is established, one can use the same bootstrap argument as in \cite{DHWZ} to complete the whole proof of Lemma \ref{lem3.2.3}, particularly deriving the estimate \eqref{3.2.27}. Thus, for brevity of presentation, in what follows we only show the uniform estimate  \eqref{3.2.28}.

Taking the inner product of \eqref{3.2.3} with $f^{\lambda,\vep}$ over $[0,T]\times\Omega\times\mathbb{R}^3$ gives that
\begin{multline}\label{3.2.29}
\int_0^T\vep\|f^{\lambda,\vep}(s)\|_{L^2}^2+\|\nu^{1/2}f^{\lambda,\vep}(s)\|_{L^2}^2+\f{1}{2}|f^{\lambda,\vep}(s)|_{L^2(\g_+)}^2\dd s\\
\leq \int_0^T\langle \lambda Kf^{\lambda}(s),f^{\lambda}(s)\rangle+\f12\big|P_{\g}f^{\lambda,\vep}(s)+r(s)\big|_{L^2(\g_-)}^2 \\
+\f{\vep}{4}\|f^{\lambda,\vep}(s)\|_{L^2}^2+\f{1}{\vep}\|g(s)\|_{L^2}^2\dd s.
\end{multline}
Note that due to the non-negativity of $L=\nu-K$,
$$
\langle \lambda Kf^{\lambda,\vep},f^{\lambda,\vep}\rangle\leq \lambda\|\nu^{1/2}f^{\lambda,\vep}\|_{L^2}^2,
$$
for any {$0\leq \lambda\leq 1$.} 
Then from \eqref{3.2.29}, we have
\begin{align}\label{3.2.30}
&\f{3\vep}{4}\int_0^T\|f^{\lambda,\vep}(s)\|_{L^2}^2\dd s+\f12\int_0^T|(I-P_{\g})f^{\lambda,\vep}(s)|_{L^2(\g_+)}^2\dd s\nonumber\\
&\leq \f{\eta}{2}\int_0^T\big|P_{\g}f^{\lambda,\vep}(s)\big|_{L^2({\g_+})}^2 +C_{\eta}\int_0^T|r(s)|_{L^{2}(\g_-)}^2\dd s+\f{1}{\vep}\int_0^T\|g(s)\|_{L^2}^2\dd s.
\end{align}
Here $\eta>0$ can be chosen to be arbitrarily small. Similar for obtaining \eqref{3.2.22}, 
we have that 
\begin{align}\label{3.2.31}
&\int_0^T|P_{\g}{f^{\lambda,\vep}}(s)|_{L^2(\g_+)}^2\dd s\nonumber\\
&\quad \leq C\int_0^T\|{e^{-\f{|v|^2}{16}}}f^{\lambda,\vep}(s)\|_{L^2}^2+|(I-P_\g )f^{\lambda,\vep}(s)|_{L^2(\g_+)}^2+\|{e^{-\f{|v|^2}{16}}}g(s)\|_{L^2}^2\dd s\notag\\
&\qquad+C\sup_{0\leq t\leq T}\|{e^{-\f{|v|^2}{16}}}f^{\lambda,\vep}(t)\|_{L^{\infty}}^2.
\end{align}
Substituting \eqref{3.2.31} into \eqref{3.2.30} gives that for any small constant $\eta>0$,
\begin{align}\label{3.2.32}
&\f{\vep}{2}\int_0^T\|f^{\lambda,\vep}(s)\|_{L^2}^2\dd s+\f14\int_0^T|(I-P_{\g})f^{\lambda,\vep}(s)|_{L^2(\g_+)}^2\dd s\nonumber\\
&\leq C\eta\sup_{0\leq t\leq T}\|{e^{-\f{|v|^2}{16}}}f^{\lambda,\vep}(t)\|_{L^{\infty}}^2+C_{\eta,\vep}\int_0^T|r(s)|_{L^{2}(\g_-)}^2\dd s+C_{\eta,\vep}\int_0^T\|g(s)\|_{L^2}^2\dd s\nonumber\\
&\leq C\eta\sup_{0\leq t\leq T}\|wf^{\lambda,\vep}(t)\|_{L^{\infty}}^2+C_{\eta,\vep}\sup_{0\leq t\leq T}\{\|\nu^{-1}wg(t)\|_{L^{\infty}}+|wr(t)|_{L^{\infty}(\g_-)}\}^2.
\end{align}
Applying the $L^{\infty}$ estimate \eqref{3.1.5} to $h^{\lambda,\vep}:=wf^{\lambda,\vep}$, we have
\begin{align}
&\sup_{0\leq t\leq T}\{\|h^{\lambda,\vep}(t)\|_{L^\infty}+|h^{\lambda,\vep}(t)|_{L^{\infty}{(\g)}}\}\nonumber \\
&\leq C\sup_{0\leq t\leq T}\{\|\nu^{-1}wg(t)\|_{L^\infty}+|wr(t)|_{L^\infty(\g_-)}\}+C\left\|f^{\lambda,\vep}\right\|_{L^2([0,T];L^2)}\nonumber\\
&\leq C{\eta^{1/2}}\sup_{0\leq t\leq T}\|h^{\lambda,\vep}(t)\|_{L^{\infty}}+{C_{\eta,\vep}}\sup_{0\leq t\leq T}\{\|\nu^{-1}wg(t)\|_{L^\infty}+|wr(t)|_{L^\infty(\g_-)}\},\nonumber
\end{align}
where we have used \eqref{3.2.32} in the second inequality. Letting $\eta>0$ be small enough, it then follows from the above estimate that
\begin{equation*}
\sup_{0\leq t\leq T}\{\|h^{\lambda,\vep}(t)\|_{L^\infty}+|h^{\lambda,\vep}(t)|_{L^{\infty}{(\g)}}\}
\leq C_{\vep}\sup_{0\leq t\leq T}\{\|\nu^{-1}wg(t)\|_{L^\infty}+|wr(t)|_{L^\infty(\g_-)}\}.
\end{equation*}
This shows \eqref{3.2.28} and then completes the proof of Lemma \ref{lem3.2.3}.
\end{proof}

\subsection{Solution to the linear inhomogeneous problem}

The last step is concerned with the limit procedure $\vep\rightarrow 0$.

\medskip
\noindent{\it Proof of Proposition \ref{prop3.2}:} Taking the inner product of \eqref{3.2.4} of $f^{\vep}$ over $[0,T]\times\Omega\times\mathbb{R}^3$, we get that for any $\eta>0$,
\begin{align}\label{3.2.34}
&\vep\int_0^T\|f^{\vep}(s)\|_{L^2}^2\dd s+\int_0^T\langle Lf^{\vep}(s), f^\vep(s)\rangle\dd s+
\f12\int_0^T|(I-P_{\g})f^{\vep}(s)|_{L^2(\g_+)}\dd s\notag\\
&\leq\eta \int_{0}^T\|\nu^{1/2}f^\vep(s)\|_{L^2}^2\dd s +\eta\int_0^T|P_{\g}f^{\vep}(s)|_{L^{2}(\g_+)}^2\dd s\notag\\
&\qquad +C_\eta\int_0^T\|\nu^{-1/2}g(s)\|_{L^2}^2\dd s+C_\eta\int_0^T|r(s)|_{L^{2}(\g_-)}^2\dd s.
\end{align}
By the coercivity estimate \eqref{c}, it holds that
$$\int_0^T\langle Lf^\vep(s), f^\vep(s)\rangle\dd s\geq c_0\int_0^T\|\nu^{1/2}(I-P)f^\vep(s)\|_{L^2}^2\dd s,
$$
where the projection $P$ is defined in \eqref{P}.
For the estimate on $P_{\g}f^{\vep}$, it is direct to see that
$$
(\pa_t+v\cdot \nabla_x)\big(e^{-\f{|v|^2}{4}}(f^{\vep})^2\big)=2e^{-\f{|v|^2}{4}}gf^{\vep}-2e^{-\f{|v|^2}{4}}f^{\vep}Lf^{\vep}
-2\vep e^{-\f{|v|^2}{4}}(f^{\vep})^2.
$$
Then it follows that
$$
\|(\pa_t+v\cdot \nabla_x)\big(e^{-\f{|v|^2}{4}}(f^{\vep})^2\big)\|_{L^1}\leq C\|{e^{-\f{|v|^2}{16}}}
g\|_{L^2}^2+C\|
{e^{-\f{|v|^2}{16}}}f^{\vep}\|_{L^2}^2.
$$
Thus, similar for obtaining \eqref{3.2.31}, it holds that
\begin{align}\label{3.2.35}
&\int_0^T|P_{\g}f^{\vep}(s)|^2_{L^2(\g_+)}\dd s\nonumber\\
&\quad \leq C\int_0^T\|{e^{-\f{|v|^2}{16}}}f^{\vep}(s)\|_{L^2}^2+|(I-P_\g )f^{\vep}(s)|_{L^2(\g_+)}^2+\|{e^{-\f{|v|^2}{16}}}g(s)\|_{L^2}^2\dd s\notag\\
&\qquad+C\sup_{0\leq t\leq T}\|{e^{-\f{|v|^2}{16}}}f^{\vep}(t)\|_{L^{\infty}}^2.
\end{align}
For the macroscopic part $Pf^{\vep}$, we note that $f^{\vep}$ satisfies the zero-mass condition \eqref{3.2.5}. Then from {\cite[Lemma 6.1]{EGKM}} 
 there exists a functional $G_{f^{\vep}}(t)$ with the property $|G_{f^{\vep}}(t)|{\lesssim\|f^\vep(t)\|_{L^2}^2}$ such that
\begin{align}\label{3.2.36}
&\int_0^t\|\nu^{1/2}Pf^{\vep}(s)\|_{L^2}^2\notag \\
&\lesssim\bigg(G_{f^{\vep}}(t)-G_{f^{\vep}}(0)\bigg)+\int_0^t\|\nu^{1/2}(I-P)f^{\vep}(s)\|_{L^2}^2\dd s\nonumber\\
&\qquad +\int_0^t\|g(s)\|_{L^2}^2\dd s
+\int_0^t|r(s)|^2_{L^2(\g_-)}\dd s+\int_0^t|(I-P_{\g})f^{{\vep}}(s)|_{L^2(\g_+)}^2\dd s.
\end{align}
In particular, taking $t=T$ in \eqref{3.2.36} and utilizing the periodicity of $f^{\vep}$, we get
\begin{align}\label{3.2.37}
&\int_0^T\|\nu^{1/2}Pf^{\vep}(s)\|_{L^2}^2\notag\\
&\leq C\int_0^T\|\nu^{1/2}(I-P)f^{\vep}(s)\|_{L^2}^2\dd s+C\int_0^T\|g(s)\|_{L^2}^2\dd s\nonumber\\
&\qquad +C\int_0^T|r(s)|_{L^2(\g_-)}^2\dd s+C\int_0^T|(I-P_{\g})f^{{\vep}}(s)|_{L^2(\g_+)}^2\dd s.
\end{align}
A suitable combination of \eqref{3.2.34}, \eqref{3.2.35} and \eqref{3.2.37} yields that
\begin{align}\label{3.2.38}
&\int_0^T\|\nu^{1/2}f^{\vep}(s)\|_{L^2}^2+|f^{\vep}(s)|_{L^2(\g_+)}^2\dd s\nonumber\\ &\leq \eta\sup_{0\leq t\leq T}\|{e^{-\f{|v|^2}{16}}}f^{\vep}(t)\|_{L^{\infty}}^2 +C_{\eta}\int_0^T\|\nu^{-1/2}g(s)\|_{L^2}^2+\|g(s)\|_{L^2}^2+|r(s)|_{L^{2}(\g_-)}^2\dd s\nonumber\\
&\leq\eta\sup_{0\leq t\leq T}\|{e^{-\f{|v|^2}{16}}}f^{\vep}(t)\|_{L^{\infty}}^2+C_{\eta}\sup_{0\leq t\leq T}\{\|\nu^{-1}wg(t)\|_{L^\infty}+|wr(t)|_{L^\infty(\g_-)}\}^2,
\end{align}
where $\eta>0$ can be chosen to be arbitrarily small. Moreover, in terms of the $L^{\infty}$ estimate \eqref{3.1.5}, it holds that
\begin{align}\label{3.2.39}
&\sup_{0\leq t\leq T}\{\|wf^{\vep}(t)\|_{L^{\infty}}+|wf^{\vep}(t)|_{L^{\infty}({\g})}\}\nonumber\\
&\leq C\sup_{0\leq t\leq T}\{\|\nu^{-1}wg(t)\|_{L^\infty}+|wr(t)|_{L^\infty(\g_-)}\}+C\|\nu^{1/2}f^{\vep}\|_{L^2([0,T];L^2)}\nonumber\\
&\leq C{\eta^{1/2}}\sup_{0\leq t\leq T}\|wf^{\vep}(t)\|_{L^{\infty}}+C_{\eta}\sup_{0\leq t\leq T}\{\|\nu^{-1}wg(t)\|_{L^\infty}+|wr(t)|_{L^\infty(\g_-)}\},
\end{align}
where we have used \eqref{3.2.38} in the last inequality. Then taking $\eta>0$ suitably small in \eqref{3.2.39}, we get the desired estimate.

To pass to the limit $\vep\rightarrow0^+$, we consider the difference
$
f^{\vep_1}-f^{\vep_2}$ with 
$0<\vep_1, \vep_2\ll 1.
$
We see that $f^{\vep_1}-f^{\vep_2}$ solves the problem:
\begin{equation}\left\{
\begin{aligned}
&\pa_t(f^{\vep_1}-f^{\vep_2})+v\cdot \nabla_x (f^{\vep_1}-f^{\vep_2})+L(f^{\vep_1}-f^{\vep_2})=\vep_2f^{\vep_2}-\vep_1f^{\vep_1},\\
&f^{\vep_1}-f^{\vep_2}|_{\g_-}=P_{\g}(f^{\vep_1}-f^{\vep_2}).\nonumber
\end{aligned}\right.
\end{equation}
Similar as before, direct energy estimates show that
\begin{align*}
&\int_0^T\|\nu^{1/2}(f^{\vep_1}-f^{\vep_2})(s)\|_{L^2}^2+|(f^{\vep_1}-f^{\vep_2})(s)|_{L^2(\g_+)}^2\dd s\nonumber\\ &\leq \eta\sup_{0\leq t\leq T}\|{e^{-\f{|v|^2}{16}}}(f^{\vep_1}-f^{\vep_2})(t)\|_{L^{\infty}}^2+C_{\eta}(\vep_1^2+\vep_2^{2})\int_0^T\|\nu^{-1/2}f^{\vep_1}(s)\|_{L^2}^2\notag\\
&\qquad\qquad\qquad\qquad\qquad\qquad +\|\nu^{-1/2}f^{\vep_2}(s)\|_{L^2}^2
+\|f^{\vep_1}(s)\|_{L^2}^2+\|f^{\vep_2}(s)\|_{L^2}^2\dd s\nonumber\\
&\leq\eta\sup_{0\leq t\leq T}\|{e^{-\f{|v|^2}{16}}}(f^{\vep_1}-f^{\vep_2})(t)\|_{L^{\infty}}^2\notag\\
&\quad +C_{\eta}(\vep_1^2+\vep_2^2)\sup_{0\leq t\leq T}\{\|wf^{\vep_1}(t)\|_{L^\infty}+\|wf^{\vep_2}(t)\|_{L^\infty}\}^2\nonumber\\
&\leq \eta\sup_{0\leq t\leq T}\|{e^{-\f{|v|^2}{16}}}(f^{\vep_1}-f^{\vep_2})(t)\|_{L^{\infty}}^2\notag\\
&\quad +C_{\eta}(\vep_1^2+\vep_2^2)\sup_{0\leq t\leq T}\{\|\nu^{-1}wg(t)\|_{L^\infty}+|wr(t)|_{L^\infty(\g_-)}\}^2.
\end{align*}
Then applying the $L^{\infty}$ estimate \eqref{3.1.5} to $h^{\vep_1}-h^{\vep_2}:=w(f^{\vep_1}-f^{\vep_2})$, we get that in the case of $0\leq\gamma\leq 1$,
\begin{align}\label{3.2.41}
&\sup_{0\leq t\leq T}\|(h^{\vep_1}-h^{\vep_2})(t)\|_{L^{\infty}}+\sup_{0\leq t\leq T}|(h^{\vep_1}-h^{\vep_2})(t)|_{L^{\infty}({\g})}\nonumber\\
&\leq C(\vep_1+\vep_2)\sup_{0\leq t\leq T}\big\{\|\nu^{-1}h^{\vep_1}(t)\|_{L^{\infty}}+\|\nu^{-1}h^{\vep_2}(t)\|_{L^{\infty}}\big\}\notag\\
&\quad +C\|\nu^{1/2}(f^{\vep_1}-f^{\vep_2})\|_{L^2([0,T];L^2)}\nonumber\\
&\leq C\eta\sup_{0\leq t\leq T}\|(h^{{\vep_1}}-h^{{\vep_2}})(t)\|_{L^{\infty}}\notag\\
&\quad +C_{\eta}(\vep_1+\vep_2)\sup_{0\leq t\leq T}\{\|\nu^{-1}wg(t)\|_{L^\infty}+|wr(t)|_{L^\infty(\g_-)}\},\nonumber\\
&\leq C(\vep_1+\vep_2)\sup_{0\leq t\leq T}\{\|\nu^{-1}wg(t)\|_{L^\infty}+|wr(t)|_{L^\infty(\g_-)}\},
\end{align}
and in the case of $-3<\gamma<0$,
\begin{align}\label{3.2.42}
&\sup_{0\leq t\leq T}\|\nu(h^{\vep_1}-h^{\vep_2})(t)\|_{L^{\infty}}+\sup_{0\leq t\leq T}|\nu(h^{\vep_1}-h^{\vep_2})(t)|_{L^{\infty}{(\g)}}\nonumber\\
&\leq C(\vep_1+\vep_2)\sup_{0\leq t\leq T}\big\{\|h^{\vep_1}(t)\|_{L^{\infty}}+\|h^{\vep_2}(t)\|_{L^{\infty}}\big\}+C\|\nu^{1/2}(f^{\vep_1}-f^{\vep_2})\|_{L^2([0,T];L^2)}\nonumber\\
&\leq C{\eta^{1/2}}\sup_{0\leq t\leq T}\|{e^{-\f{|v|^2}{16}}}(h^{{\vep_1}}-h^{{\vep_2}})(t)\|_{L^{\infty}}\notag\\&\quad +C_{\eta}(\vep_1+\vep_2)\sup_{0\leq t\leq T}\{\|\nu^{-1}wg(t)\|_{L^\infty}+|wr(t)|_{L^\infty(\g_-)}\}\nonumber\\
&\leq C(\vep_1+\vep_2)\sup_{0\leq t\leq T}\{\|\nu^{-1}wg(t)\|_{L^\infty}+|wr(t)|_{L^\infty(\g_-)}\},
\end{align}
{by taking $\eta>0$ suitably small. }Therefore, from \eqref{3.2.41} and \eqref{3.2.42} we have respectively shown  that $f^{\vep}$ is Cauchy in $L^{\infty}_{w}$ for $0\leq \gamma\leq 1$, and Cauchy in $L^{\infty}_{\nu w}$ for $-3<\gamma<0$. Let  $f(t,x,v)$ be the limit function of $f^{\vep}(t,x,v)$ in the corresponding function space. It is direct to check that  $f(t,x,v)$  satisfies \eqref{3.0.1}. Finally, the time-periodicity and continuity of $f$ directly follow from the time-periodicity and continuity of $f^{\vep}$. The proof of Proposition \ref{prop3.2} is therefore complete. \qed

\subsection{Proof of Theorem \ref{thm1.1}}
We consider the solution sequence $\{f^j(t,x,v)\}$ iteratively solved from
\begin{align*}
\begin{cases}
\pa_t f^{j+1}+v\cdot \nabla_x f^{j+1}+Lf^{j+1}={-}L_{\sqrt{\mu}
{f^*}}f^j+\Gamma(f^j,f^j),\\[1.5mm]
f^{j+1}|_{\gamma_-}=P_\gamma f^{j+1}+\frac{\mu_{\theta}-
\mu}{\sqrt{\mu}}\int_{v'\cdot n(x)>0}f^j\sqrt{\mu} \{v'\cdot n(x)\} \dd v' +r,
\end{cases}
\end{align*}
for $j=0,1,2\cdots$, where  we have set $f^0\equiv0$. Here we have denoted
$$
r(t,x,v)=\frac{\mu_{\theta}-\mu_{\b{\theta}}}{\sqrt{\mu}}\int_{v'\cdot n(x)>0} {F^{*}(x,v')}
 \{v'\cdot n(x)\} \dd v',
$$
and
$$
{L_{\sqrt{\mu}f^{*}}f^j=-\f{1}{\sqrt{\mu}}[Q(\sqrt{\mu}f^*,\sqrt{\mu}f^j)+Q(\sqrt{\mu}f^j,\sqrt{\mu}f^*)].}
$$
A direct calculation shows that
\begin{align}\label{3.3.2}
\int_{\Omega\times\mathbb{R}^3} \Gamma(f^j,f^j)\sqrt{\mu(v)} \dd v\dd x=\int_{\Omega\times\mathbb{R}^3}
L_{\sqrt{\mu}{f^*}
}f^j\sqrt{\mu(v)} \dd v\dd x=0,
\end{align}
and
\begin{equation}\label{3.3.3}
\int_{v\cdot n(x)<0} [\mu_\theta(v)-
\mu(v)] \{v\cdot n(x)\} \dd v =\int_{v\cdot n(x)<0} [\mu_{\theta}(v)-\mu_{\b{\theta}}(v)] \{v\cdot n(x)\} \dd v=0.
\end{equation}
Furthermore, one can verify that
\begin{align}\label{3.3.4}
\|\nu^{-1}wL_{\sqrt{\mu}{f^*}}
f^j\|_{L^{\infty}}+\|\nu^{-1} w \Gamma(f^j,f^j)\|_{L^\infty}\leq C \delta\|wf^j\|_{L^\infty}+C\|wf^j\|^2_{L^\infty},
\end{align}
and
\begin{equation}\label{3.3.5}
\left| w\bigg\{r +\frac{\mu_\theta-
\mu}{\sqrt{\mu}}\int_{v'\cdot n(x)>0} f^j\sqrt{\mu} \{v'\cdot n(x)\} \dd v'\bigg\} \right|_{L^\infty{(\g_-)}}
\leq C\delta_1+C\delta |f^j|_{L^\infty(\g_+)}.
\end{equation}
Recall \eqref{3.3.2}, \eqref{3.3.3}, \eqref{3.3.4} and \eqref{3.3.5}. Then, by applying \eqref{3.0.4} to $f^{j+1}$, we get
\begin{multline}\label{3.3.6}
\sup_{0\leq s\leq T}\{\|wf^{j+1}(s)\|_{L^\infty}+|wf^{j+1}(s)|_{L^\infty{(\g)}}\}\\
\leq C\delta_1+C\sup_{0\leq s\leq T}\big\{\|f^j(s)\|_{L^\infty}^2+\delta\|wf^j(s)\|_{L^\infty}+\delta|wf^j(s)|_{L^\infty(\g_+)}\big\}.
\end{multline}
From \eqref{3.3.6}, it is direct to prove by an induction argument that
\begin{equation}\label{3.3.7}
\sup_{0\leq s\leq T}\|wf^{j}(s)\|_{L^\infty}+\sup_{0\leq s\leq T}|wf^{j}(s)|_{L^\infty{(\g)}}\leq 2C\delta_1,
\end{equation}
for $j=1,2,\cdots$, provided that $\delta>0$ is suitably small, where $C$ is a generic constant independent of $j$. For the convergence of the approximation sequence  $f^j$, we consider the difference $f^{j+1}-f^j$ which satisfies
\begin{multline}
\pa_t(f^{j+1}-f^j)+v\cdot \nabla_x (f^{j+1}-f^j)+L(f^{j+1}-f^j)\nonumber\\
={-}L_{\sqrt{\mu}{f^*}
}(f^j-f^{j-1})+\Gamma(f^{j}-f^{j-1},f^{j})+\Gamma(f^{j-1},f^{j}-f^{j-1}),\nonumber
\end{multline}
with the boundary condition
\begin{align}
(f^{j+1}-f^j)|_{\gamma_-}&=P_\gamma (f^{j+1}-f^j)\notag\\
&\quad+\frac{\mu_\theta-\mu}{\sqrt{\mu}}\int_{v'\cdot n(x)>0} (f^j-f^{j-1})\sqrt{\mu} \{v'\cdot n(x)\} \dd v'.\nonumber
\end{align}
Once again, applying \eqref{3.0.4} to $f^{j+1}-f^j$ gives that
\begin{align}\label{3.3.8}
&\sup_{0\leq s\leq T}\big\{\|w(f^{j+1}-f^j)(s)\|_{L^\infty}+|w(f^{j+1}-f^j)(s)|_{L^\infty{(\g)}}\big\}\nonumber\\
&\leq C\bigg(\delta+\sup_{0\leq s\leq T}\big\{\|wf^{j}(s)\|_{L^\infty}+\|wf^{j-1}(s)\|_{L^\infty}\}\bigg)\nonumber\\
&\quad\times \sup_{0\leq s\leq T}\big\{\|w(f^{j}-f^{j-1})(s)\|_{L^\infty}+|w(f^j-f^{j-1})(s)|_{L^\infty(\gamma_+)}\Big\}\nonumber\\
&\leq C\delta \sup_{0\leq s\leq T}\big\{\|w(f^{j}-f^{j-1})(s)\|_{L^\infty}+|w(f^j-f^{j-1})(s)|_{L^\infty(\g_+)}\big\}\nonumber\\
&\leq \f12\sup_{0\leq s \leq T}\big\{\|w(f^{j}-f^{j-1})(s)\|_{L^\infty}+|w(f^j-f^{j-1})(s)|_{L^\infty(\g_+)}\big\},
\end{align}
where we have used \eqref{3.3.7} in the second inequality and also we have taken $\delta>0$ small enough such that $C\delta\leq 
1/2$. Hence, $f^j(t,x,v)$ is a Cauchy sequence in $L^\infty_w$. Let
$
f^{{per}}(t,x,v)=\lim_{j\rightarrow\infty} f^j(t,x,v)
$
in $L^\infty_w$.
It is direct to check that
$$
F^{per}(t,x,v)=\mu+\sqrt{\mu} f^{{per}}(t,x,v)
$$
is the time-periodic solution to the boundary-value problem \eqref{1.1} and \eqref{1.7}, and also \eqref{add.thmcon} and \eqref{1.5} are satisfied. The proof of \eqref{1.4} for the non-negativity of  $F^{per}(t,x,v)$ is left to the next section. The uniqueness and  continuity of $f^{{per}}(t,x,v)$ can be obtained in a usual way, cf.~\cite{DHWZ}. Therefore this completes  the proof of Theorem \ref{thm1.1}. \qed

\section{Asymptotical stability}

This section is concerned with the large-time behavior of solutions to the initial-boundary value problem \eqref{ibvp} whenever $F_0(x,v)$ is sufficiently close to $F^{per}(0,x,v)$ at initial time. As a byproduct, the result about the dynamical stability of the non-trivial time-periodic profile $F^{per}(t,x,v)$ in turn yields its non-negativity.

As for obtaining the existence of the time-periodic solution $F^{per}(t,x,v)$, we need to first study the linear inhomogeneous problem in the following Proposition \ref{prop4.1}. As its proof is is more or less the same as the one of {\cite[Proposition 7.1]{EGKM}} 
 for $0\leq \gamma\leq 1$ and \cite[Proposition 4.4]{DHWZ} 
for $-3<\gamma<0$. The full details are omitted for brevity.

\begin{proposition}\label{prop4.1}
Let $-3<\gamma\leq 1, 0
{<} q<\f18$ and $\beta>\max\{3,3-\gamma\}$. Let
$$
\|wf_0\|_{L^{\infty}}+\|\nu^{-1}wg\|_{L^{\infty}}<\infty,
$$
and
\begin{align*}
\int_{\Omega}\int_{\mathbb{R}^3} f_0(x,v)\sqrt{\mu(v)}\,\dd x\dd v=\int_{\Omega} \int_{\mathbb{R}^3} g(t,x,v)\sqrt{\mu(v)}\,{\dd x\dd v}=0.
\end{align*}
Then if
$$
\sup_{0\leq t\leq T}|\theta(t,\cdot)-1|_{L^\infty(\pa\Omega)}
$$ is sufficiently small, the linear inhomogeneous initial-boundary value problem:
\begin{equation}\left\{
\begin{aligned}
&\pa_tf+ v\cdot \nabla_xf+Lf=g,\quad t>0,x\in\Omega, v\in \mathbb{R}^3,\\
&f(t,x,v)|_{\gamma_{-}}=P_{\gamma}f+{\f{\mu_{\theta}-{\mu}}{\sqrt{\mu}}\int_{v'\cdot n(x)>0}f\sqrt{\mu}\{n(x)\cdot v'\}\,\dd v'},\nonumber\\
&f(t,x,v)|_{t=0}=f_0(x,v),
\end{aligned}\right.
\end{equation}
admits a unique solution $f(t,x,v)$ satisfying that
\begin{multline}\label{4.2}
\sup_{0\leq s\leq t}e^{c {s}^\rho}\{\|wf(t)\|_{L^{\infty}}+|wf(t)|_{L^{\infty}{(\g)}}\}\\
{\leq  C\|wf_0\|_{L^{\infty}}+C\sup_{0\leq s\leq t}e^{c s^\rho}\|\nu^{-1}wg(s)\|_{L^{\infty}}},
\end{multline}
for any $t>0$, where $\rho$ is defined in \eqref{def.rho}, and $c>0$ is a generic small constant.
Moreover, if $\Omega$ is convex, $f_0(x,v)$ is continuous except on $\g_0$, $g$ is continuous in the interior of $[0,\infty)\times\Omega\times \mathbb{R}^3$,
\begin{align*}
f_0(x,v)|_{\gamma_-}=P_{\g}f_0+{\f{\mu_{\theta}-{\mu}}{\sqrt{\mu}}\int_{v'\cdot n(x)>0}f_0\sqrt{\mu}\{n(x)\cdot v'\}\dd v'},
\end{align*}
and $\theta(t,x) $ is continuous over $\mathbb{R}\times\partial \Omega$, then the solution $f(t,x,v)$ is also continuous over $[0,\infty)\times \{\bar{\Omega}\times \mathbb{R}^{3}\setminus\g_0\}$.
\end{proposition}

\noindent{\it Proof of Theorem \ref{thm1.2}:} We construct the solution via the following iteration:
\begin{align*}
\begin{cases}
\dis \pa_t f^{j+1}+v\cdot \nabla_x f^{j+1}+Lf^{j+1}={-}L_{\sqrt{\mu}f^{per}}f^j+\Gamma(f^j,f^j),\\
\dis f^{j+1}|_{\gamma_-}=P_\gamma f^{j+1}+\frac{\mu_{\theta}-
\mu}{\sqrt{\mu}}\int_{v'\cdot n(x)>0}f^{{j+1}}\sqrt{\mu} \{v'\cdot n(x)\} \dd v',\\
f^{j+1}(0,x,v)=f_0(x,v),
\end{cases}
\end{align*}
for $j=0,1,2\cdots$, where we have set $f^0\equiv0$, and also 
$$
L_{\sqrt{\mu}f^{{per}}}f^j:={-\f{1}{\sqrt{\mu}}[Q(\sqrt{\mu}f^{per},\sqrt{\mu}f)+Q(\sqrt{\mu}f,\sqrt{\mu}f^{per})]}.
$$
Similar for obtaining estimates \eqref{3.3.2}-\eqref{3.3.5}, we have
\begin{align}
&\int_{\Omega\times\mathbb{R}^3} \Gamma(f^j,f^j)\sqrt{\mu(v)} \dd v\dd x=\int_{\Omega\times\mathbb{R}^3}
L_{\sqrt{\mu}f^{per}}f^j\sqrt{\mu(v)} \dd v\dd x=0,\nonumber
\end{align}
{and
\begin{align}
\|\nu^{-1}w[L_{\sqrt{\mu}f^{per}}f^j\|_{L^\infty}+\|\nu^{-1}w\Gamma(f^j,f^j)]\|_{L^{\infty}}\leq C\delta'\|wf^j\|_{L^{\infty}}+C\|wf^j\|^2_{L^{\infty}}.\nonumber
\end{align}}
Then we can apply the linear time-decay property \eqref{4.2} to $f^{j+1}$ to obtain that
\begin{align}\label{4.5}
&\sup_{0\leq s\leq t}e^{c s^{\rho}}\{\|wf^{j+1}(s)\|_{L^{\infty}}+|wf^{j+1}(s)|_{L^{\infty}(\g)}\}\nonumber\\
&\leq C \|wf_0\|_{L^{\infty}}+C\delta'\sup_{0\leq s\leq t}e^{c s^\rho}\|wf^j(s)\|_{L^{\infty}}+C\sup_{0\leq s\leq t}e^{c s^{\rho}}\|wf^j(s)\|_{L^{\infty}}^2.
\end{align}
From \eqref{4.5}, we can also use the induction argument to show that
$$
\sup_{0\leq s\leq t}e^{c s^{\rho}}\{\|wf^{j+1}(s)\|_{L^{\infty}}+|wf^{j+1}(s)|_{L^{\infty}(\g)}\}\leq 2C \|wf_0\|_{L^{\infty}},
$$
provided that both $\delta'>0$ and $\|wf_0\|_{L^{\infty}}$ are suitably small. Similar to obtain \eqref{3.3.8}, one can show that $\{f^j\}_{j=1}^{\infty}$ is a Cauchy sequence {in $L^\infty_w$}, then we obtain the solution $f(t,x,v)$ as the limit of $f^j(t,x,v)$.
The uniqueness and continuity is standard, and the positivity can be shown by the same method as in \cite{EGKM}. Therefore, we complete the proof of Theorem \ref{thm1.2}.\qed

\medskip
\noindent{\bf Acknowledgments.}  Renjun Duan is partially supported by the General Research Fund (Project No.~14302817). Yong Wang is partly supported by NSFC Grant No.~11771429, 11688101, and 11671237.


\end{document}